\documentclass[UTF8]{article}
\usepackage[margin=1.5in]{geometry}  
\usepackage{graphicx}
\usepackage{amssymb}             
\usepackage{amsmath}
\usepackage{mathtools}
\usepackage{slashed}            
\usepackage{amsfonts}              
\usepackage{amsthm}
\usepackage{breqn}
\usepackage{verbatim}              
\usepackage{enumitem}
\usepackage[colorlinks,linkcolor=black,anchorcolor=black,citecolor=black]{hyperref}

\usepackage[numbers]{natbib}
\usepackage{mathtools}
\usepackage{cases}
\usepackage{fancyhdr}

\theoremstyle{definition}
\newtheorem{thm}{Theorem}[section]
\newtheorem{nota}[thm]{Notation}

\newtheorem{defn}{Definition}[section]
\newtheorem{lem}[thm]{Lemma}
\newtheorem{prop}[thm]{Proposition}

\newtheorem*{rmk}{Remark}

\newcommand{\R}{\mathbb{R}}

\newcommand{\T}{\mathbb{T}}
\newcommand{\TP}{\overline{\partial}}
\newcommand{\TL}{\overline{\Delta}}
\newcommand{\curl}{\text{curl }}
\newcommand{\dive}{\text{div }}

\newcommand{\q}{\quad}
\newcommand{\p}{\partial}

\newcommand{\DD}{\mathcal{D}}

\newcommand{\nab}{\nabla}

\newcommand{\lap}{\Delta}

\newcommand{\beq}{\overline{\epsilon}}
\newcommand{\di}{\text{div}\,}

\newcommand{\cp}{\overline{\partial}{}}

\newcommand{\vol}{\text{vol}\,}

\newcommand{\PP}{\mathcal{P}}

\newcommand{\tl}{\tilde}
\newcommand{\bt}{\bar{\theta}}
\numberwithin{equation}{section}
\usepackage{xcolor}

\begin{document}
\title{A Regularity Result for the Incompressible Magnetohydrodynamics Equations with Free Surface Boundary}
\author{Chenyun Luo\thanks{Vanderbilt University, Nashville, TN, USA.
\texttt{chenyun.luo@vanderbilt.edu}}\,\, and Junyan Zhang \thanks{Johns Hopkins University, MD, USA.
\texttt{zhang.junyan@jhu.edu}} }
\maketitle

\begin{abstract}
We consider the three-dimensional incompressible magnetohydrodynamics (MHD) equations in a bounded domain with small volume and free moving surface boundary. We establish a priori estimate for solutions with minimal regularity assumptions on the initial data in Lagrangian coordinates.  In particular, due to the lack of the Cauchy invariance for MHD equations,  the smallness assumption on the fluid domain is required to compensate a loss of control of the flow map. Moreover,  we show that the magnetic field has certain regularizing effect which allows us to control the vorticity of the fluid  and that of the magnetic field. To the best of our knowledge this is the first result that focuses on the low regularity solution for incompressible free-boundary MHD equations. 
\end{abstract}
\tableofcontents
\section{Introduction}
The goal of this manuscript is to investigate the solutions in low regularity Sobolev spaces for the following incompressible inviscid MHD equations in a moving domain: 
\begin{equation}
\begin{cases}
\partial_t u+u\cdot\nabla u-B\cdot\nabla B+\nabla (p+\frac{1}{2}|B|^2)=0,~~~& \text{in}~\DD; \\
\partial_t B+u\cdot\nabla B-B\cdot\nabla u=0,~~~&\text{in}~\DD; \\
\dive u=0,~~\dive B=0,~~~&\text{in}~\DD,
\end{cases}\label{MHD}
\end{equation}
describing the motion of conducting fluids in an electromagnetic field, where $\DD={\cup}_{0\leq t\leq T}\{t\}\times \Omega(t)$ and $\Omega(t)\subset \R^3$ is the domain occupied by the fluid with small volume whose boundary $\p\Omega(t)$ moves with the velocity of the fluid.  Under this setting, the fluid velocity $u=(u_1,u_2,u_3)$, the magnetic field $B=(B_1,B_2,B_3)$, the fluid pressure $p$ and the domain $\DD$ are to be determined; in other words, given a simply connected bounded domain $\Omega(0)\subset \R^3$ and the initial data $u_0$ and $B_0$ satisfying the constraints $\di u_0=0$ and $\di B_0=0$, we want to find a set $\DD$ and the vector fields $u$ and $B$ solving \eqref{MHD} satisfying the initial conditions: 
\begin{equation}
\Omega(0)=\{x:  (0,x)\in \DD\},\q (u,B)=(u_0, B_0),\q \text{in}\,\,\{0\}\times \Omega_0.
\end{equation}
We also require the following boundary conditions on the free boundary $\p\DD={\cup}_{0\leq t\leq T}\{t\}\times \p\Omega(t)$: 
\begin{equation}
\begin{cases}
(\p_t+u\cdot \nabla )|_{\partial\DD}\in \mathcal{T}(\partial\DD) \\
p=0~~~&\text{on}~\partial\DD, \\
|B|=c,~~B\cdot \mathcal{N}=0~~~&\text{on}~\partial\DD,
\end{cases}\label{MHDB}
\end{equation} 
where $\mathcal{T}(\p\DD)$ is the tangent bundle of $\p\DD$, $\mathcal{N}$ is the exterior unit normal to $\p\Omega_t$ and $c\geq 0$ is a constant. The first condition of \eqref{MHDB} means that the boundary moves with the velocity of the fluid, the second condition of \eqref{MHDB} means that the region outside $\Omega_t$ is vacuum, where $B\cdot\mathcal{N}=0$ on $\p\Omega_t$ implies that the fluid is a perfect conductor; in other words, the induced electric field $\mathcal{E}$ satisfies $\mathcal{E}\times \mathcal{N}=0$ on $\p\Omega_t$. Also, the condition $|B|=c$ on $\p\Omega_t$ yields that the physical energy is conserved, i.e., denoting $D_t=\p_t+u\cdot \nab$, and invoking the divergence free condition for both $u$ and $B$, we have: 
\[
\begin{aligned}
&~~~~~\frac{d}{dt}\Big[\frac{1}{2}\int_{\Omega(t)} |u|^2 + \frac{1}{2}\int_{\Omega(t)} |B|^2\Big]\\
&=\int_{\Omega(t)} u\cdot D_tu+\int_{\Omega(t)}B\cdot D_t B\\
&=-\int_{\Omega(t)}u\cdot \nab (p+\frac{1}{2}|B|^2)+\int_{\Omega(t)}u\cdot (B\cdot \nab B) +\int_{\Omega(t)}B\cdot (B\cdot \nab u)\\
&=-\int_{\p\Omega(t)}(u\cdot \mathcal{N})p-\underbrace{\int_{\p\Omega(t)} \frac{1}{2}(u\cdot \mathcal{N})c^2}_{=0 \,\,\text{by Gauss theorem}}+\int_{\Omega(t)}u\cdot (B\cdot \nab B)-\int_{\Omega(t)}u\cdot (B\cdot \nab B)=0 .
\end{aligned}
\]

We will establish a priori bounds for the MHD equations \eqref{MHD}-\eqref{MHDB} when $u_0,B_0\in H^{2.5+\delta}(\Omega(0))$ for $\delta\in (0,0.5)$  under the physical sign condition
\begin{equation}
-\nabla_{\mathcal{N}} (p+\frac{1}{2}|B|^2)\geq \epsilon_0>0 \q \text{on}\,\,\p\Omega(t). 
\label{taylor}
\end{equation}
 We recall here that for the free-boundary problem of the motion of a incompressible fluid without magnetic field (i.e., the incompressible free-boundary Euler equations), the physical sign condition reads 
\begin{equation*}
-\nab_{\mathcal{N}} p \geq \epsilon_0>0\q \text{on}\,\,\p\Omega(t).
\end{equation*}
Condition \eqref{taylor} was first discovered by Hao and Luo \cite{hao2014priori} when proving the a priori energy estimate for the free boundary incompressible MHD equations with $H^4$ initial data. Very recently, they proved that \eqref{MHD}-\eqref{MHDB} is ill-posed when \eqref{taylor} is violated \cite{hao2018ill}. The quantity $p+\frac{1}{2}|B|^2$ (i.e., the total pressure) plays an important role here in our analysis. In fact, it determines the acceleration of the moving surface boundary.  

\subsection{History and background}
In the absence of the magnetic field $B$, the system \eqref{MHD} is reduced to the free-boundary Euler equations which has attracted much attention in the past two decades. Important progress has been made for both incompressible and compressible flows, with or without surface tension, and with or without vorticity. Without attempting to be exhaustive, we refer \cite{alazard2014cauchy, christodoulou2000motion, coutand2007LWP,   disconzi2017prioriI, ignatova2016local, kukavica2017local, lindblad2002, lindblad2005well,  lindblad2018priori, luo2017motion, nalimov1974cauchy,  SchweizerFreeEuler, shatah2008geometry, shatah2008priori, shatah2011local, wu1997LWPww, wu1999LWPww, zhang2008free} for more details. 

On the other hand, the MHD equations describe the behavior of an electrically conducting fluid (e.g., a plasma) acted on by a magnetic field.
In particular, the free-boundary MHD equations (also known as the plasma-interface problem) describes the phenomenon when the plasma is separated from the outside wall by a vacuum, whose motion can be formulated as the  incompressible free-boundary MHD equations. 

Although the MHD equations in a fixed domain have been the focus of a great deal of activities, e.g., \cite{ cao2010two, chen2008existence, friedlander2011global,  friedlander1990nonlinear,  friedlander1995stability,  jiang2010incompressible,  wu2011global}, much less is known for the free-boundary case. The main difficulty is the strong coupling between $u$ and $B$ (i.e., the appearance of $B\cdot \nab B$ and $B\cdot \nab u$ terms in the first and second equations of \eqref{MHD}, respectively). In fact, the appearance of the Lorentzian force term $B\cdot \nab B$ destories the Cauchy invariance, which provides good estimates for $\curl v$ when $B$ is absent; indeed, one can see this by commuting the curl operator through the first equation of \eqref{MHD}, which implies\footnote{We refer \eqref{cauchy invariance fail1}-\eqref{cauchy invariance fail 2} for the detailed computation.}
$$
(\p_t+\nab_u) \curl u \sim \nab (\curl B).
$$
   Nevertheless, it is remarkable that the magnetic field $B$ yields certain regularizing effect (cf. \cite{wangyanjin2012}), which can be derived from the transport equation of $B$ (i.e., the second equation of \eqref{MHD}). Such regularizing effect plays an important role to control the full Sobolev norms of $\curl B$ and $\curl u$ and hence the full Sobolev norm of $B$ and $u$ via the div-curl estimate. We will provide more details on this  in Section \ref{section 1.3}.

For the free-boundary MHD equations, the local (in time) well-posedness (LWP) of the linearized equations was studied by Morando-Trakhinin-Trebeschi \cite{morando2014well}, Secchi-Trakhinin \cite{secchi2011well} and Trakhinin \cite{trakhinin2010well}. For the nonlinear equations, Hao-Luo \cite{hao2014priori} proved the a priori energy estimate with $H^4$ initial data and the LWP was established by Secchi-Trakhinin \cite{secchi2013well} and Gu-Wang \cite{gu2016construction}. Also, we mention here that in Hao \cite{hao2016motion} and Sun-Wang-Zhang \cite{sun2017well}, the authors studied the a priori energy estimate and LWP, respectively, for the free-boundary MHD equations with nontrivial vacuum magnetic field. 

In this manuscript, we establish the local a priori energy estimate with $u, B\in H^{2.5+\delta}$ with $\delta>0$ is arbitrary. This agrees with the minimal regularity assumption (i.e., $H^{\frac{d}{2}+1+\delta}$, where $d$ is the spatial dimension) that one may expect for the velocity field in the theory of the free-boundary incompressible Euler equations (see, e.g.,  \cite{DKT, kukavica2014regularity, kukavica2017local}). In fact,  Bourgain-Li \cite{bourgain2015strong} proved that the incompressible Euler equations with $H^{\frac{d}{2}+1}$ initial data are ill-posed even in the free space $\R^d$.

\subsection{MHD system in Lagrangian coordinates and Main result}
We reformulate the MHD equations in Lagrangian coordinates, in which  the free domain becomes fixed. Let $\Omega$ be a bounded domain in $\R^3$. Denoting coordinates on $\Omega$ by $y=(y_1,y_2,y_3)$, we define $\eta: [0,T]\times \Omega\to\DD$ to be the flow map of the velocity $u$, i.e., 
\begin{equation}
\p_t \eta (t,y)=u(t,\eta(t,y)),\q
\eta(0,y)=y.
\end{equation}
We introduce the Lagrangian velocity, magnetic field and fluid pressure, respectively, by
\begin{equation}
v(t,y)=u(t,\eta(t,y)),\q
b(t,y)=B(t,\eta(t,y)),\q
q(t,y)=p(t,\eta(t,y)).
\end{equation}
Let $\p$ be the spatial derivative with respect to $y$ variable. We introduce the cofactor matrix $a=[\p\eta]^{-1}$, which is well-defined since $\eta(t,\cdot)$ is almost the identity map when $t$ is sufficiently small. It's worth noting that $a$ verifies the Piola's identity, i.e., 
\begin{equation}
\p_\mu a^{\mu\alpha} = 0.
\label{piola}
\end{equation}
Here, the summation convention is used for repeated upper and lower indices, and in above and throughout, all indices (e.g., Greek and Latin) range over $1,2,3$.

Denote the total pressure $p_{\text{total}}=p+\frac{1}{2}|B|^2$ and let $Q=p_{\text{total}}(t,\eta(t,y))$. Then \eqref{MHD}-\eqref{MHDB} can be reformulated as: 
\begin{equation}
\begin{cases}
\partial_tv_{\alpha}-b_{\beta}a^{\mu\beta}\partial_{\mu}b_{\alpha}+a^{\mu}_{\alpha}\partial_{\mu}Q=0~~~& \text{in}~[0,T]\times\Omega;\\
\partial_t b_{\alpha}-b_{\beta}a^{\mu\beta}\partial_{\mu}v_{\alpha}=0~~~&\text{in}~[0,T]\times \Omega ;\\
a^{\mu}_{\alpha}\partial_{\mu}v^{\alpha}=0,~~a^{\mu}_{\alpha}\partial_{\mu}b^{\alpha}=0~~~&\text{in}~[0,T]\times\Omega;\\
v_3=0~~~&\text{on}~\Gamma_0;\\
a^{\mu\nu}b_{\mu}b_{\nu}=c^2,~Q=\frac{1}{2}c^2,~ a^{\mu}_{\nu}b^{\nu}N_{\mu}=0 ~~~&\text{on}~\Gamma_;\\
\frac{\partial Q}{\partial N}=\partial_3 Q\leq -\epsilon_0~~~&\text{on}~\Gamma_1.
\end{cases}\label{MHDL}
\end{equation}
\rmk 
In above and throughout, the upper index of $a$ represents the number of the rows whereas the lower index represents the number of the columns, i.e., $a^{\text{row}}_{\text{column}}$. 
\\

For the sake of simplicity and clean notation, here we consider the model case when \begin{equation}
\Omega=\T^2\times (0,\beq),
\label{Omega}
\end{equation}
 where $\beq \ll 1$ and $ \partial\Omega=\Gamma_0\cup\Gamma_1$ and $\Gamma_1=\T^2\times\{\beq\}$ is the top (moving) boundary, $\Gamma_0=\T^2\times\{0\}$ is the fixed bottom. We shall treat the general bounded domains with small volume in Section 6 by adapting what has been done in \cite{DKT}. However, choosing $\Omega$ as above allows us to focus on the real issues of the problem without being distracted by the cumbersomeness of the partition of unity. Let $N$ stands for the outward unit normal of $\p\Omega$. In particular, we have $N=(0,0,-1)$ on $\Gamma_0$ and $N=(0,0,1)$ on $\Gamma_1$. 

In this paper, we prove: 

\begin{thm}\label{MHDthm}
Let $\Omega$ be defined as in \eqref{Omega}. Let $(\eta, v, b)$ be the solution of \eqref{MHDL} and $\delta\in(0,0.5)$. Assume that $v(0,\cdot)=v_0\in H^{2.5+\delta}(\Omega)$ and $b(0,\cdot)=b_0\in H^{2.5+\delta}(\Omega)$ be divergence free vector fields and $b_0\cdot N=0$ on $\p\Omega$. Let
\begin{equation}\label{N(t)}
N(t): =\|\eta(t)\|_{H^{3+\delta}}^2+\|v(t)\|_{H^{2.5+\delta}}^2+\|b(t)\|_{H^{2.5+\delta}}^2.
\end{equation}
Then for sufficiently small $\beq$, there exists a $T>0$, depending only on $N(0)$ and $\beq$ such that $N(t)\leq P(N(0))$ for all $t\in [0,T]$,  provided the physical sign condition
\begin{equation}
-\frac{\p Q}{\p N}\big|_{t=0} = -\p_3 Q \big|_{t=0} \geq \epsilon_0>0,\q\text{on}\,\,\Gamma_1
\end{equation}
holds. Here, $P$ is a polynomial of its arguments. 
\end{thm}

\rmk We will show that the physical sign condition \eqref{taylor0} propagates within $[0,T]$. In other words, it holds
\begin{equation}\label{taylor0}
-\p_3 Q(t) \geq \epsilon_0>0,\q\text{on}\,\,\Gamma_1,\q t\in [0,T].
\end{equation}

\subsection{Strategy, organisation of the paper, and discussion of the difficulties}\label{section 1.3}
\paragraph*{Notations.} All definitions and notations will be defined as they are introduced. In addition, a list of symbols will be given at the end of this section for a quick reference.
\defn The $L^2$-  based Sobolev spaces are denoted by $H^s(\Omega)$, where
we abbreviate corresponding norm $\|\cdot\|_{H^r(\Omega)}$ as $\|\cdot\|_{H^r}$ when no confusion can arise. We denote by $H^s(\Gamma)$ the Sobolev space of functions defined on $\Gamma$, with norm $\|\cdot\|_{H^{s}(\Gamma)}$. 

\nota We use $\epsilon$ to denote a small positive constant which may vary from expression to expression. Typically, $\epsilon$ comes from choosing sufficiently small time, from Lemma \ref{estimatesofa} and from the Young's inequality. 
\nota We use $P=P(\cdots)$ to denote a generic polynomial in its arguments.\\

Now we can state the strategies we used and discuss the discovery and the difficulty in MHD system.

\noindent\textbf{Gronwall-Type argument and div-curl estimates}

The proof of Theorem \ref{MHDthm} relies on div-curl type estimates of the velocity field $v$, the magnetic field $b$ and the Lagrangian flow map $\eta$.  In particular, let $N(t)$ be defined as in Theorem \ref{MHDthm}. Then if $\vol(\Omega)$ is sufficiently small (i.e., $\beq\ll 1$), there exists a $T>0$ such that the estimate
\begin{equation}
N(t)\lesssim M_0+ \epsilon P(N(t))+P(N(t))\int_0^t P(N(s))ds
\label{intro close}
\end{equation}
holds whenever $t\in [0,T]$, where $M_0=M_0(\|v_0\|_{H^{2.5+\delta}}, \|b_0\|_{H^{2.5+\delta}})$ .  This implies $N(t)\lesssim M_0$ by a Gronwall-type argument that can be found in Chapter 1 of Tao \cite{tao2006nonlinear}.\\

\noindent\textbf{Creation of vorticity by the magnetic field}

The vorticity of the conducting fluid cannot be controlled analogously to that in the case of a non-conducting fluid due to the lack of the Cauchy invariance, since its derivation involves  the derivative of the Lorentzian force $(b_0\cdot\p)b$, which contributes to higher order terms. In particular, let $\epsilon^{\mu\nu\tau}$ be the anti-symmetric tensor with $\epsilon^{123}=1$. We have: 
\begin{equation}
\begin{aligned}
\p_t(\epsilon^{\mu\nu\tau}\p_\nu v^m\p_\tau \eta_m)= \underbrace{\epsilon^{\mu\nu\tau}\p_\nu v^m\p_\tau v_m}_{=0}+\epsilon^{\mu\nu\tau}\p_\nu v_t^m\p_\tau \eta_m\\
=\underbrace{-\epsilon^{\mu\nu\tau}\p_\nu (a^{\ell m} \p_\ell Q)\p_\tau \eta_m}_{=0,\,\, \text{same as the Euler's equations}}+\epsilon^{\mu\nu\tau}\p_\nu (b_0^\sigma \p_\sigma b^m)\p_\tau \eta_m, 
\end{aligned}\label{cauchy invariance fail1}
\end{equation}
where the last term in the second line is equal to
\begin{align}
\epsilon^{\mu\nu\tau}\p_\nu (b_0^\sigma \p_\sigma b^m)\delta^\tau_m+\epsilon^{\mu\nu\tau}\p_\nu (b_0^\sigma \p_\sigma b^m)(\p_\tau \eta_m-\delta^\tau_m)\nonumber\\
=\curl (b_0^\sigma\p_\sigma b)+\epsilon^{\mu\nu\tau}\p_\nu (b_0^\sigma \p_\sigma b^m)(\p_\tau \eta_m-\delta^\tau_m),
\label{cauchy invariance fail 2}
\end{align}
is nonzero in general. We remark here that it is the Lorentzian force that causes the strong coupling between $v$ and $b$. One can imagine that the Lorentzian force twists the trajectory of an electric particle in a magnetic field and produces vorticity even if the initial data is curl-free. However, we can control $\curl v$ and $\curl b$ from their evolution equation derived by taking the Eulerian curl operator to the first equation of \eqref{MHDL}. This will be dicussed in the following paragraph. \\

\noindent\textbf{Regularizing effect of $b$:  Controlling $\curl v,$ $\curl b$ and pressure $Q$}

The key to control $\|v\|_{H^{2.5+\delta}}$ and $\|b\|_{H^{2.5+\delta}}$ is to control $\|B_a v\|_{H^{1.5+\delta}}$ and $\|B_a b\|_{H^{1.5+\delta}}$, where $B_a$ denotes the Eulerian curl operator, i.e., $[B_a X]_\lambda=\epsilon_{\lambda\tau\alpha}a^{\mu\tau}\p_{\mu}X^{\alpha}$, where $\epsilon_{\lambda\tau\alpha}$ is the anti-symmetric tensor with $\epsilon_{123}=1$. These quantities are treated  straightforwardly for non-conducting fluids (i.e., Euler equations) thanks to the remarkable Cauchy invariance. We, nevertheless, have to control them differently since the Cauchy invariance fails for MHD equations due to the presence of the Lorentzian force term $b_\beta a^{\mu\beta} \p_\mu b$.  Inspired by Gu-Wang \cite{gu2016construction}, one can derive the evolution equation for $B_a v$ and $B_a b$. With the help of the following identities\footnote{We refer Lemma \ref{GW} for the detailed derivation.}
\begin{equation}
b_{\beta}a^{\mu\beta}=b_0^{\mu}\q \text{and}\q b_{\mu}=(b_0\cdot\p)\eta_\mu  \label{stabilizing 1}
\end{equation} mentioned in Gu-Wang \cite{gu2016construction}, one can rewrite the first equation of \eqref{MHDL} as
\begin{equation}
\p_t v_\alpha + a^{\mu}_{\alpha} \p_\mu Q = (b_0\cdot \p)^2 \eta_\alpha.
\label{re-write}
\end{equation}

Now, one may apply the curl operator $B_a$ on both sides of \eqref{re-write} and get: 
\begin{equation}
(B_a \p_tv)_{\lambda}=(B_a ((b_0\cdot\p)^2\eta))_{\lambda},
\label{curl_e}
\end{equation} 
which yields an evolution equation after commuting $\p_t$ and $b_0\cdot\p$  on both sides of \eqref{curl_e}: 
\begin{equation}
\p_t(B_a v)_{\lambda}-(b_0\cdot\p)B_a ((b_0\cdot\p)\eta)_{\lambda}=\text{error terms + commutators}.
\end{equation}
This, in particular, yields an energy identity for $B_a v$ and $B_a b=B_a (b_0\cdot \p)\eta$, i.e., 
\begin{equation}
E_{\curl}(t): =\frac{1}{2}\int_{\Omega}|\p^{1.5+\delta} B_a v|^2+|\p^{1.5+\delta} B_a (b_0\cdot\p)\eta|^2,
\end{equation}
and it can be shown that $\mathcal{E}(t)$ verifies the following estimates by using Kato-Ponce inequalities \eqref{KATO}
\begin{equation}
E_{\curl}(t)\leq \|b_0\|_{H^{2.5+\delta}}+ \int_0^t P(\|\eta\|_{H^{2.5+\delta}}, \|v\|_{H^{2.5+\delta}} , \|b\|_{H^{2.5+\delta}} ). 
\end{equation}



On the other hand, it is worth pointing out that the control of $\|Q\|_{H^{3+\delta}}$ and $\|\p_3Q_t\|_{L^\infty(\p\Omega)}$ (and hence $\|Q_t\|_{H^{2.5+\delta}}$) are both required. These quantities are needed even for the incompressible free-boundary Euler equations, whose a priori energy estimate can be closed by requiring $\eta$ to be half derivatives more regular than $v$ (see, e.g., \cite{alazard2014cauchy, kukavica2014regularity, kukavica2017local}). In the case of a conducting fluid, i.e., MHD equations, we have to use the regularizing effect of the magnetic field (i.e., identities \eqref{stabilizing 1}) to show that $\|\eta\|_{H^{3+\delta}}$ is still good enough to control $\|Q\|_{H^{3+\delta}}$ and $\|Q_t\|_{H^{2.5+\delta}}$. In particular, $Q_t$ satisfies an elliptic equation that involves $b_0^\mu \p_\mu a^{\nu\alpha}\p_\nu\p_t b_\alpha$ as part of its source term, whose $H^{0.5+\delta}$ norm requires $\eta\in H^{3.5+\delta}$ to control. However, this term can be avoided by invoking the identities \eqref{stabilizing 1} when deriving the elliptic PDE of $Q_t$.
\rmk One may drop the requirement for $\|\eta\|_{H^{s+0.5}}$ when $s>3.5$ using Alinhac's good unknowns thanks to the fact that $\p a\in L^\infty$. We refer \cite{gu2016construction, hao2014priori} for details.
\\


\noindent\textbf{Smallness of the fluid's volume is required:  Nonlinear control of $\curl\eta$}

One needs to control $\|\curl \eta\|_{H^{2+\delta}}$ (and hence $\|B_a\eta\|_{H^{2+\delta}}$) to close the a priori estimate. This can be done in the case of a non-conducting fluid using the Cauchy invariance if one assumes $\omega_0=\curl v_0 \in H^{2+\delta}$ (cf. \cite{kukavica2017local}). This, again, fails for MHD equations. 
In order to control $B_a\p\eta$, one can only hope to use the multiplicative Sobolev inequality and Young's inequality with $\epsilon$ to derive the nonlinear estimate, which produces a term $\epsilon^{-1}P(\|\eta(0)\|_{H^{2.5+\delta}})$. Therefore, we  require  the body of the conducting fluid to have small volume to fight the growth of the vorticity brought by twisting effect of the Lorentzian force (in other words, the strong coupling between $b$ and $v$), otherwise the Gronwall-type argument no longer holds since it requires $\epsilon$ to be sufficiently small. The smallness of the fluid body can be propagated\footnote{One may also choose to add an articifical smoothness conditions for $\eta$ (e.g., $\eta\in H^{3+\delta}(\Omega)$). But such conditions do not seem to be the ones that can be propagated.} if it holds initially since $\eta$ is volume-preseving. 
\\

\noindent\textbf{Organization of the paper: }

The manuscript will be organized as follows. In Section \ref{section 2} we record the preliminary estimates for the cofactor matrix $a$ and its time derivatives. Also, the well-known Kato-Ponce commutator estimates are summarized  as Lemma \ref{KatoPonce} for readers' convenience.  Section \ref{section 3} is devoted to control $\|Q\|_{H^{3+\delta}}$ and $\|Q_t\|_{H^{2.5+\delta}}$, which is required for the tangential estimate of $v$. 
In Section \ref{section 4} we prove the tangential estimates for both $v$ and $b$. Finally, in Section \ref{section 5}, we provide the control for the full Sobolev norms of $v$, $b$ and $\eta$ using a div-curl type estimate. Also, we show that the physical sign condition \eqref{taylor0} propagates within a short period by showing that the quantity $\p_3 Q|_{\Gamma_1}$ is $1/4$-H\"older continuous in time, which allows us to close the a priori estimates.

\paragraph*{Acknowledgement: } We would like to thank Marcelo Disconzi and Igor Kukavica for many insightful discussions. Also, we thank Igor for sharing his idea on the proof of Lemma \ref{Igor}.

\paragraph*{List of symbols: }
\begin{itemize}
\item $\epsilon$:   A small positive constant which may vary from expression to expression.
\item $\beq$:  The ``height" of the fluid domain $\Omega$, which is also chosen to be sufficiently small.
\item $a=[\p \eta]^{-1}$:  The cofactor matrix;
\item $\|\cdot\|_{H^s}$:  We denote $\|f\|_{H^s}: = \|f(t,\cdot)\|_{H^s(\Omega)}$ for any function $f(t,y)\text{ on }[0,T]\times\Omega$.
\item $P$:  A generic polynomial in its arguments;
\item $\PP$:  $\PP=P(\|v\|_{H^{2.5+\delta}}, \|b\|_{H^{2.5+\delta}})$ (and so $\PP_0=P(\|v_0\|_{H^{2.5+\delta}}, \|b_0\|_{H^{2.5+\delta}})$;
\item $N(t)$:  $N(t)=\|\eta\|_{H^{3+\delta}}^2+\|v\|_{H^{2.5+\delta}}^2+\|b\|_{H^{2.5+\delta}}^2$;
\item $\cp = (I-\overline{\Delta})^{1/2}$ where $\overline{\Delta}=\p_1^2+\p_2^2$, and $S=\cp^{2.5+\delta}$:  Tangential differential operators. 
\end{itemize}
\begin{flushright}
$\square$
\end{flushright}

\section{Preliminary Lemmas}
\label{section 2}

The first lemma is about some basic estimate of the cofactor matrix $a$, which shall be used throughout the rest of the manuscript. 

\begin{lem}\label{estimatesofa}
Suppose $\|\p v\|_{L^{\infty}([0,T];H^{1.5+\delta}(\Omega))}\leq M$. If $T\leq \frac{1}{CM}$ for a sufficiently large constant $K$, then the following estimates hold: 

(1) $\|\p\eta\|_{H^{1.5+\delta}(\Omega)}\leq C$ for $t\in [0,T]$;

(2) $\det(\p\eta(t,x))=1$ for $(x,t)\in\Omega\times [0,T]$;

(3) $\|a(\cdot,t)\|_{H^{1.5+\delta}(\Omega)}\leq C$ for $t\in [0,T]$;

(4) $\|a_t(\cdot,t)\|_{L^p(\Omega)}\leq C\|\p v\|_{L^p(\Omega)}$ for $t\in [0,T]$, $1\leq p\leq\infty$;

(5) $\|a_t(\cdot,t)\|_{H^r(\Omega)}\leq C\|\p v\|_{H^r(\Omega)}$ for $t\in [0,T]$, $0\leq r\leq 1.5+\delta$;

(6) $\|a_{tt}(\cdot,t)\|_{H^r(\Omega)}\leq C\|\p v\|_{H^{1.5+\delta}}\|\p v\|_{H^r}+C\|\p v_t\|_{H^r}$, for $t\in [0,T]$, $0< r\leq 0.5+\delta$;

(7) For every $0<\epsilon\leq 1$, there exists a constant $C>0$ such that for all $0\leq t\leq T': =\min\{\frac{\epsilon}{CM},T\}>0$, we have $$\|a^{\mu}_{\nu}-\delta^{\mu}_{\nu}\|_{H^{1.5+\delta}(\Omega)}\leq\epsilon, ~~\|a^{\mu}_{\alpha}a^{\nu}_{\alpha}-\delta^{\mu\nu}\|_{H^{1.5+\delta}(\Omega)}\leq \epsilon.$$

(8) $\p_m a_{\alpha}^\mu=-a^{\mu}_{\nu}\p_{\beta}\p_m \eta^\nu a^{\beta}_{\alpha}$ for $m=1,2,3.$
\end{lem}
\begin{proof}
See \cite{kukavica2014regularity}:  (1)-(7) is Lemma 3.1 and (8) is formula (6.6).
\end{proof}

The next lemma reveals the regularizing  effect of the magnetic field $b$; in particular,  the flow map $\eta$ is more regular in the direction of $b_0$. It was also used in Wang \cite{wangyanjin2012} and Gu-Wang \cite{gu2016construction}

\begin{lem}\label{GW}
Let $(v,b,\eta)$ be a solution to \eqref{MHDL} with initial data $(v_0,b_0,\eta_0)$. Then the following two identities hold: 
\begin{equation}\label{GW1}
a^{\nu\alpha} b_{\alpha}=b_0^{\nu},
\end{equation} 
\begin{equation}\label{GW2}
b^{\beta}=(b_0\cdot\p)\eta^{\beta}=b_0^{\nu}\p_{\nu}\eta^{\beta}.
\end{equation} 
\end{lem}
\begin{proof}
For \eqref{GW1}, we multiply $a^{\nu\alpha}$ to the second equation of \eqref{MHDL} to get $$a^{\nu\alpha}\p_tb_{\alpha}=a^{\nu\alpha}b_{\beta}a^{\mu\beta}\p_{\mu}\p_t\eta_{\alpha}=a^{\nu\alpha}b_{\beta}\p_t(\underbrace{a^{\mu\beta}\p_{\mu}\eta_{\alpha}}_{=\delta^{\beta}_{\alpha}})-b_{\beta}\p_ta^{\mu\beta}(\underbrace{\p_{\mu}\eta_{\alpha}a^{\nu\alpha}}_{\delta_{\mu}^{\nu}})=-b_{\alpha}\p_ta^{\nu\alpha},$$ so $\p_t(a^{\nu\alpha}b_{\alpha})=0$ and thus $a^{\nu\alpha}b_{\alpha}=b_0^{\nu}$. For \eqref{GW2}, it can be easily derived by multiplying $\p_{\nu}\eta_{\beta}$ on the both sides of \eqref{GW1} and using $a: \p\eta =I$.
\end{proof}

The last lemma records the well-known Kato-Ponce commutator estimates, the proof of which can be found in \cite{kato1988commutator}.
\begin{lem}\label{KatoPonce}
Let $J=(I-\Delta)^{1/2}$, $s\geq 0$. Then the following estimates hold: 

(1) $\forall s\geq 0$ and $1<p<\infty$, we have
\begin{equation}\label{KATO}
\|J^s(fg)-f(J^s g)\|_{L^p}\lesssim \|\partial f\|_{L^{\infty}}\|J^{s-1} g\|_{L^p}+\|J^s f\|_{L^p}\|g\|_{L^{\infty}};
\end{equation}

(2) $\forall s\geq 0$, we have 
\begin{equation}\label{product}
\|J^s(fg)\|_{L^2}\lesssim \|f\|_{W^{s,p_1}}\|g\|_{L^{p_2}}+\|f\|_{L^{q_1}}\|g\|_{W^{s,q_2}},
\end{equation}with $1/2=1/p_1+1/p_2=1/q_1+1/q_2$ and $2\leq p_1,q_2\leq \infty$;

(3) $\forall s\in (0,1)$, we have 
\begin{equation}\label{kato1}
\|J^s(fg)-f(J^s g)-(J^s f)g\|_{L^p}\lesssim \| f\|_{W^{s_1,p_1}}\|g\|_{W^{s-s_1,p_2}},
\end{equation}
where $0<s_1<s$ and $1/p_1+1/p_2=1/p$ with $1<p<p_1,p_2<\infty$;

(4) $\forall s\geq 1$, we have
\begin{equation}\label{kato2}
\|J^s(fg)-f(J^s g)\|_{L^2}\lesssim \|f\|_{W^{s,p_1}}\|g\|_{L^{p_2}}+\|f\|_{W^{1,q_1}}\|g\|_{W^{s-1,q_2}},
\end{equation}
where $1/2=1/p_1+1/q_1=1/p_2+1/q_2$ with $1<p<p_1,p_2<\infty$; and 
\begin{equation}\label{kato3}
\|J^s(fg)-(J^sf)g-f(J^sg)\|_{L^p}\lesssim\|f\|_{W^{1,p_1}}\|g\|_{W^{s-1,q_2}}+\|f\|_{W^{s-1,q_1}}\|g\|_{W^{1,q_2}}
\end{equation} for all the $1<p<p_1,p_2,q_1,q_2<\infty$ with $1/p_1+1/p_2=1/q_1+1/q_2=1/p$.
\end{lem}

\begin{flushright}
$\square$
\end{flushright}

\section{Pressure Estimates}
\label{section 3}
In this section we derive the estimates for $\|Q\|_{H^{3+\delta}}$ and $\|Q_t\|_{H^{2.5+\delta}}$. These quantities are both required in Section \ref{section 4}. 
\nota We denote $\PP=P(\|v\|_{H^{2.5+\delta}}, \|b\|_{H^{2.5+\delta}})$ and so $\PP_0=P(\|v_0\|_{H^{2.5+\delta}}, \|b_0\|_{H^{2.5+\delta}})$. 
\begin{lem}
Assume Lemma \ref{estimatesofa} holds. Then the total pressure $Q$ satisfies: 
\begin{equation}\label{estimatesofQ}
\|Q\|_{H^{3+\delta}}\lesssim \PP_0+\PP+P(\|\eta\|_{H^{3+\delta}})\left(\|Q_0\|_{H^{2+\delta}}+\int_0^t\|Q_t\|_{H^{2+\delta}}\right),
\end{equation}
and its time derivative $Q_t$ satisfies: 
\begin{equation}\label{estimatesofQt}
\|Q_t\|_{H^{2.5+\delta}}\lesssim\PP_0+\PP+P(\|v\|_{H^{2.5+\delta}})\left(\|Q_0\|_{H^{2+\delta}}+\int_0^t\|Q_t\|_{H^{2+\delta}}\right).
\end{equation}
\end{lem}

\textbf{Proof: } Applying $a^{\nu\alpha}\partial_{\nu}$ to the first equation of \eqref{MHDL}, we have: 
\begin{equation}
a^{\nu\alpha}\partial_{\nu}(a^{\mu}_{\alpha}\partial_{\mu}Q)=-a^{\nu\alpha}\partial_{\nu}\partial_t  v_{\alpha}+a^{\nu\alpha}\partial_{\nu}(b_{\beta}a^{\mu\beta}\partial_{\mu}b_{\alpha})=-a^{\nu\alpha}\partial_{\nu}\partial_tv_{\alpha}+a^{\nu\alpha}\partial_{\nu}(b_0^{\mu}\partial_{\mu}b_{\alpha}),
\end{equation} where we have used \eqref{GW1}.

Invoking the Piola's identity \eqref{piola}, Lemma \ref{estimatesofa} (8) and \eqref{GW2}, we get: 
\begin{equation*}
-a^{\nu\alpha}\p_\nu\p_t v_\alpha=\p_t a^{\nu\alpha}\p_\nu v_\alpha,
\end{equation*}
and
\[
\begin{aligned}
a^{\nu\alpha}\p_{\nu}(b_0^{\mu}\p_{\mu}b_{\alpha})&=a^{\nu\alpha}\p_{\nu}b_0^{\mu}\p_{\mu}b_{\alpha}+a^{\nu\alpha}b_0^{\mu}\p_{\nu}\p_{\mu}b_{\alpha} \\
&=a^{\nu\alpha}\p_{\nu}b_0^{\mu}\p_{\mu}b_{\alpha}+b_0^{\mu}\p_{\mu}(\underbrace{a^{\nu\alpha}\p_{\nu}b_{\alpha}}_{=0})-b_0^{\mu}\p_{\mu}a^{\nu\alpha}\p_{\nu}b_{\alpha}\\
&=a^{\nu\alpha}\p_{\nu}b_0^{\mu}\p_{\mu}b_{\alpha}+b_0^{\mu}\p_{\mu}\p_{\beta}\eta_{\gamma}a^{\nu\gamma}a^{\beta\alpha}\p_{\nu}b_{\alpha}\\
&=a^{\nu\alpha}\p_{\nu}b_0^{\mu}\p_{\mu}b_{\alpha}+\p_{\beta}((b_0\cdot\p)\eta_{\gamma})a^{\nu\gamma}a^{\beta\alpha}\p_{\nu}b_{\alpha}-\p_{\beta}b_0^{\mu}\underbrace{\p_{\mu}\eta_{\gamma}a^{\nu\gamma}}_{=\delta_{\mu}^{\nu}}a^{\beta\alpha}\p_{\nu}b_{\alpha}\\
&=a^{\nu\alpha}\p_{\nu}b_0^{\mu}\p_{\mu}b_{\alpha}+\p_{\beta}b_{\gamma}a^{\nu\gamma}a^{\beta\alpha}\p_{\nu}b_{\alpha}-\p_{\beta}b_0^{\mu}a^{\beta\alpha}\p_{\mu}b_{\alpha}.
\end{aligned}
\]

Thus, the total pressure $Q$ satisfies
\begin{equation}
\p^{\mu}\p_{\mu}Q=\partial_ta^{\nu\alpha}\p_{\nu}  v_{\alpha}+\p_{\nu}((\delta^{\mu\nu}-a^{\mu}_{\alpha}a^{\nu\alpha})\p_{\mu}Q)+a^{\nu\alpha}\p_{\nu}b_0^{\mu}\p_{\mu}b_{\alpha}+\p_{\beta}b_{\gamma}a^{\nu\gamma}a^{\beta\alpha}\p_{\nu}b_{\alpha}-\p_{\beta}b_0^{\mu}a^{\beta\alpha}\p_{\mu}b_{\alpha}, \label{Q}
\end{equation} 
 with the boundary conditions
\begin{equation}\label{QB1}
Q=\frac{1}{2}c^2 \,\,\text{on}\,\,\Gamma_1,\q  \text{and} \q
a^{\mu}_{\alpha}\p_{\mu}QN^{\alpha}=0\,\, \text{on}\,\, \Gamma_0,
\end{equation} 
where the second condition can be rewritten as
\begin{equation}\label{QB2}
\partial_{\alpha}QN^{\alpha}=(\delta^{\mu}_{\alpha}-a^{\mu}_{\alpha})\partial_{\mu}Q N^{\alpha}\,\, \text{on}\,\, \Gamma_0.
\end{equation}

The standard elliptic estimate yields that 
\begin{equation}\label{Q1}
\begin{aligned}
\|Q\|_{H^{3+\delta}}&\lesssim \underbrace{\|\p_t a^{\nu\alpha}\p_{\nu}  v_{\alpha}\|_{H^{1+\delta}}}_{\mathcal{Q}_1}\\
&+\underbrace{\|(\delta^{\mu\nu}-a^{\mu}_{\alpha}a^{\nu\alpha})\p_{\mu}Q\|_{H^{2+\delta}}+\|(\delta_{\alpha}^{\mu}-a^{\mu}_{\alpha})\p_{\mu}Q N^{\alpha}\|_{H^{1.5+\delta}(\Gamma)}}_{\mathcal{Q}_2}\\
&+\underbrace{\|a^{\nu\alpha}\p_{\nu}b_0^{\mu}\p_{\mu}b_{\alpha}\|_{H^{1+\delta}}+\|\p_{\beta}b_{\gamma}a^{\nu\gamma}a^{\beta\alpha}\p_{\nu}b_{\alpha}\|_{H^{1+\delta}}+\|\p_{\beta}b_0^{\mu}a^{\beta\alpha}\p_{\mu}b_{\alpha}\|_{H^{1+\delta}}}_{\mathcal{Q}_3}\\
\end{aligned}
\end{equation}
\paragraph*{Bounds for $\mathcal{Q}_1$:  } We have: 
\begin{equation}\label{Q11}
\begin{aligned}
\|\p_ta^{\nu\alpha}\p_{\nu}  v_{\alpha}\|_{H^{1+\delta}}&\lesssim \|\p_ta^{\nu\alpha}\|_{H^{1+\delta}}\|\p_{\nu}  v_{\alpha}\|_{H^{1.5+\delta}}\\
&\lesssim  \|\eta\|_{H^{2.5+\delta}}\|v\|_{H^{2+\delta}}\|v\|_{H^{2.5+\delta}}\leq C\|v\|_{H^{2.5+\delta}}^2\|v\|_{H^{2+\delta}}^2,
\end{aligned}
\end{equation}
where we used $\|a\|_{H^{1.5+\delta}}\lesssim \|\eta\|_{H^{2.5+\delta}}^2$ and the multiplicative Sobolev inequality
\begin{equation}
\|fg\|_{H^{1+\delta}}\lesssim\|f\|_{H^{1+\delta}}\|g\|_{H^{1.5+\delta}},
\label{multiplicative}
\end{equation}
which is a direct consequence of \eqref{product} and the  Sobolev embedding.

\paragraph*{Bounds for $\mathcal{Q}_2$:  }
Invoking Lemma \ref{estimatesofa} (7) and \eqref{product}, we have: 

\begin{equation}\label{Q21}
\begin{aligned}
\|(\delta^{\mu\nu}-a^{\mu}_{\alpha}a^{\nu\alpha})\p_{\mu}Q\|_{H^{2+\delta}}&\lesssim \|I-a: a^T\|_{L^{\infty}}\|\p_{\mu}Q\|_{H^{2+\delta}}+\|I-a: a^T\|_{H^{2+\delta}}\|\p_{\mu}Q\|_{L^{\infty}} \\
&\lesssim\epsilon\|Q\|_{H^{3+\delta}}+(1+\|\eta\|_{H^{3+\delta}}^4)\|Q\|_{H^{2+\delta}}^{1/2}\|Q\|_{H^{3+\delta}}^{1/2}\\
&\lesssim \epsilon\|Q\|_{H^{3+\delta}}+P(\|\eta\|_{H^{3+\delta}})\|Q\|_{H^{2+\delta}}\\
&\lesssim \epsilon \|Q\|_{H^{3+\delta}}+ P(\|\eta\|_{H^{3+\delta}})\left(\|Q_0\|_{H^{2+\delta}}+\int_0^t \|Q_t\|_{H^{2+\delta}}ds\right),
\end{aligned}
\end{equation}
and similarly
\begin{equation}\label{Q22}
\begin{aligned}
\|(\delta^{\mu}_{\alpha}-a^{\mu}_{\alpha})\p_{\mu}Q N^{\alpha}\|_{H^{1.5+\delta}(\Gamma)}&\lesssim\|I-a\|_{L^{\infty}}\|Q\|_{H^{3+\delta}}+\|I-a\|_{H^{2+\delta}}\|Q\|_{H^{2.5+\delta}} \\
&\lesssim\epsilon \|Q\|_{H^{3+\delta}}+ P(\|\eta\|_{H^{3+\delta}})\left(\|Q_0\|_{H^{2+\delta}}+\int_0^t \|Q_t\|_{H^{2+\delta}}ds\right).
\end{aligned}
\end{equation}

\paragraph*{Bounds for $\mathcal{Q}_3$:  } All the terms in $\mathcal{Q}_2$ can be controlled by $C\|b\|_{H^{2.5+\delta}}\|b_0\|_{H^{2.5+\delta}}+C\|b\|_{H^{2.5+\delta}}^2$ via the multiplicative Sobolev inequality. We only write the first term and the others are treated similarly.
\begin{equation}\label{Q3}
\|a^{\nu\alpha}\p_{\nu}b_0^{\mu}\p_{\mu}b_{\alpha}\|_{H^{1+\delta}}\lesssim\|a^{\nu\alpha}\|_{H^{1.5+\delta}}\|\p_{\nu}b_0^{\mu}\p_{\mu}b_{\alpha}\|_{H^{1+\delta}}\lesssim C\|b\|_{H^{2.5+\delta}}\|b_0\|_{H^{2.5+\delta}}.
\end{equation} 
Summing up the bounds for $\mathcal{Q}_1$-$\mathcal{Q}_3$, then absorbing the $\epsilon$-term to LHS, we conclude the estimates of $Q$ as: 
\begin{equation}
\|Q\|_{H^{3+\delta}}\lesssim \PP_0+\PP+P(\|\eta\|_{H^{3+\delta}})\left(\|Q_0\|_{H^{2+\delta}}+\int_0^t \|Q_t\|_{H^{2+\delta}}ds\right).
\label{estimatesofQpre}
\end{equation}

Now we start to prove the estimates of $Q_t$. Taking time derivative of \eqref{Q}, we obtain: 
\begin{equation}\label{Qt0}
\begin{aligned}
\p^{\mu}\p_{\mu}Q_t&=\p_{tt}a^{\nu\alpha}\p_{\nu}v_{\alpha}+\p_ta^{\nu\alpha}\p_{\nu}\p_t v_{\alpha}\\
&~~~~-\p_{\nu}(\p_ta^{\mu}_{\alpha}a^{\nu\alpha}\p_{\mu}Q)-\p_{\nu}(a^{\mu}_{\alpha}\p_t a^{\nu\alpha}\p_{\mu}Q)+\p_{\nu}((\delta^{\mu\nu}-a^{\mu}_{\alpha}a^{\nu}_{\alpha})\p_{\mu}Q_t) \\
&~~~~+a^{\nu\alpha}_t\p_{\nu}b_0^{\mu}\p_{\mu}b_{\alpha}+a^{\nu\alpha}\p_{\nu}b_0^{\mu}\p_t\p_{\mu}b_{\alpha}+\p_t(\p_{\beta}b_{\gamma}\p_{\nu}b_{\alpha})a^{\nu\gamma}a^{\beta\alpha}+\p_{\beta}b_{\gamma}\p_t(a^{\nu\gamma}a^{\beta\alpha})\p_{\nu}b_{\alpha}\\
&~~~~-\p_{\beta}b_0^{\mu}a^{\beta\alpha}\p_t\p_{\mu}b_{\alpha}-\p_{\beta}b_0^{\mu}a^{\beta\alpha}_t\p_{\mu}b_{\alpha}.
\end{aligned}
\end{equation}
with the boundary conditions
\begin{equation}
\begin{aligned}
Q_t=0\q\text{on}\,\,\Gamma_1,\\
\partial_{\alpha}Q_tN^{\alpha}=-\partial_t a^{\mu}_{\alpha}\partial_{\mu}Q N^{\alpha}+(\delta^{\mu}_{\alpha}-a^{\mu}_{\alpha})\partial_{\mu}Q_t N^{\alpha}\q\text{on}\,\,\Gamma_{0}.
\end{aligned}
\end{equation}
By the elliptic estimate, we have: 
\begin{equation}
\begin{aligned}\label{Qt1}
&~~~~\|Q_t\|_{H^{2.5+\delta}}\\
&\lesssim \|\p_{tt}a^{\nu\alpha}\p_{\nu}v_{\alpha}\|_{H^{0.5+\delta}}+\|\p_ta^{\nu\alpha}\p_{\nu}\p_t v_{\alpha}\|_{H^{0.5+\delta}}+\|\p_ta^{\mu}_{\alpha}a^{\nu\alpha}\p_{\mu}Q\|_{H^{1.5+\delta}}+\|a^{\mu}_{\alpha}\p_t a^{\nu\alpha}\p_{\mu}Q\|_{H^{1.5+\delta}}\\
&~~~~+\|(\delta^{\mu\nu}-a^{\mu}_{\alpha}a^{\nu}_{\alpha})\p_{\mu}Q_t\|_{H^{1.5+\delta}}+\|\p_t a^{\mu}_{\alpha}\p_{\mu}Q N^{\alpha}\|_{H^{{1+\delta}}(\Gamma)}+\|(\delta^{\mu}_{\alpha}-a^{\mu}_{\alpha})\p_{\mu}Q_t N^{\alpha}\|_{H^{1+\delta}(\Gamma)}\\
&~~~~+\|a^{\nu\alpha}_t\p_{\nu}b_0^{\mu}\p_{\mu}b_{\alpha}\|_{H^{0.5+\delta}}+\|a^{\nu\alpha}\p_{\nu}b_0^{\mu}\p_t\p_{\mu}b_{\alpha}\|_{H^{0.5+\delta}}+\|\p_t(\p_{\beta}b_{\gamma}\p_{\nu}b_{\alpha})a^{\nu\gamma}a^{\beta\alpha}\|_{H^{0.5+\delta}}\\
&~~~~+\|\p_{\beta}b_{\gamma}\p_t(a^{\nu\gamma}a^{\beta\alpha})\p_{\nu}b_{\alpha}\|_{H^{0.5+\delta}}\\
&~~~~+\|\p_{\beta}b_0^{\mu}a^{\beta\alpha}\p_t\p_{\mu}b_{\alpha}\|_{H^{0.5+\delta}}+\|\p_{\beta}b_0^{\mu}a^{\beta\alpha}_t\p_{\mu}b_{\alpha}\|_{H^{0.5+\delta}}.
\end{aligned}
\end{equation}
First, since $\partial_t v_{\alpha}=a^{\mu}_{\alpha}\p_{\mu}Q-b_{\beta}a^{\mu\beta}\partial_{\mu}b_{\alpha}$ we have: 
\begin{equation}
\|v_t\|_{H^{1.5+\delta}}\lesssim \|b\|_{H^{1.5+\delta}}^2\|a\|_{H^{1.5+\delta}}+\|Q\|_{H^{2.5+\delta}}\|a\|_{H^{1.5+\delta}}.
\end{equation}
Using this and the multiplicative Sobolev inequality
\begin{equation}
\|fg\|_{H^{0.5+\delta}}\lesssim\|f\|_{H^{0.5+\delta}}\|g\|_{H^{1.5+\delta}},
\label{multiplicative2}
\end{equation}
the first two terms of \eqref{Qt1} are treated as: 
\begin{equation}\label{Qt111}
\begin{aligned}
&~~~~\|\p_{tt}a^{\nu\alpha}\p_{\nu}v_{\alpha}\|_{H^{0.5+\delta}}+\|\p_ta^{\nu\alpha}\p_{\nu}\p_t v_{\alpha}\|_{H^{0.5+\delta}}\\
&\lesssim \|a_{tt}\|_{H^{0.5+\delta}}\|v\|_{H^{2.5+\delta}}+\|a_t\|_{H^{1.5+\delta}}\|v_t\|_{H^{1.5+\delta}}\\
&\lesssim \|v\|_{H^{2.5+\delta}}^2\|v\|_{H^{1.5+\delta}} + \|\eta\|_{H^{2.5+\delta}}\|v\|_{H^{2.5+\delta}}\|v_t\|_{H^{1.5+\delta}}  \\
&\lesssim \PP+\|v\|_{H^{2.5+\delta}}\left(\|Q_0\|_{H^{2+\delta}}+\int_0^t \|Q_t\|_{H^{2+\delta}}ds\right).
\end{aligned}
\end{equation}

Second, invoking \eqref{multiplicative} and Lemma \ref{estimatesofa} (7),  the terms containing $Q$ in \eqref{Qt1} are treated as: 
\begin{equation}\label{Qt12}
\begin{aligned}
&~~~~~~~~\|\p_ta^{\mu}_{\alpha}a^{\nu\alpha}\p_{\mu}Q\|_{H^{1.5+\delta}}+\|a^{\mu}_{\alpha}\p_t a^{\nu\alpha}\p_{\mu}Q\|_{H^{1.5+\delta}}+\|\p_t a^{\mu}_{\alpha}\p_{\mu}Q N^{\alpha}\|_{H^{{1+\delta}}(\Gamma)}\\
&~~~~+\|(\delta^{\mu\nu}-a^{\mu}_{\alpha}a^{\nu}_{\alpha})\p_{\mu}Q_t\|_{H^{1.5+\delta}}+\|(\delta^{\mu}_{\alpha}-a^{\mu}_{\alpha})\p_{\mu}Q_t N^{\alpha}\|_{H^{1+\delta}(\Gamma)} \\
&\lesssim\|a\|_{H^{1.5+\delta}}\|a_t\|_{H^{1.5+\delta}}\|Q\|_{H^{2.5+\delta}}\\
&~~~~+\|a_t\|_{H^{1.5+\delta}}\|Q\|_{H^{2.5+\delta}}+\|I-a^T: a\|_{H^{1.5+\delta}}\|Q_t\|_{H^{2.5+\delta}}+\|I-a\|_{H^{1.5+\delta}}\|Q_t\|_{H^{2.5+\delta}}\\
&\lesssim \|v\|_{H^{2.5+\delta}}\left(\|Q_0\|_{H^{2+\delta}}+\int_0^t \|Q_t\|_{H^{2+\delta}}ds\right)+\epsilon\|Q_t\|_{H^{2.5+\delta}},
\end{aligned}
\end{equation}
which can be controlled appropriately by the RHS of \eqref{estimatesofQt} by plugging in the estimate \eqref{estimatesofQ}.

Now it remains to control the terms containing $b$ in \eqref{Qt1} (the last 6 terms). In fact, all the terms containing $b$ can be controlled with the help of the multiplicative Sobolev inequality \eqref{multiplicative2}. The terms not containing $b_t$ are easier to control: 
\begin{equation}\label{sb}
\begin{aligned}
&~~~~\|a^{\nu\alpha}_t\p_{\nu}b_0^{\mu}\p_{\mu}b_{\alpha}\|_{H^{0.5+\delta}}+\|\p_{\beta}b_{\gamma}\p_t(a^{\nu\gamma}a^{\beta\alpha})\p_{\nu}b_{\alpha}\|_{H^{0.5+\delta}}+\|\p_{\beta}b_0^{\mu}a^{\beta\alpha}_t\p_{\mu}b_{\alpha}\|_{H^{0.5+\delta}}\\
&\lesssim \|a_t\|_{H^{0.5+\delta}}\|b_0\|_{H^{2.5+\delta}}\|b\|_{H^{2.5+\delta}}\\
&~~~~+\|a_t\|_{H^{0.5+\delta}}\|a\|_{H^{1.5+\delta}}\|b\|_{H^{2.5+\delta}}^2+\|a_t\|_{H^{0.5+\delta}}\|b_0\|_{H^{2.5+\delta}}\|b\|_{H^{2.5+\delta}}\|\eta\|_{H^{2.5+\delta}}\\
&\lesssim \PP.
\end{aligned}
\end{equation}

For the terms containing $b_t$, we have to put $H^{0.5+\delta}$ norm on $\p b_t$ when we use the multiplicative Sobolev inequality \eqref{multiplicative2}, because we only have $b_t\in H^{1.5+\delta}$. This can be directly derived by taking time derivative of $\p_t b_{\alpha}=b_{\beta}a^{\mu\beta}\partial_{\mu}v_{\alpha}=b_0^{\mu}\partial_{\mu}v_{\alpha}$, which implies $$\|b_t\|_{H^{1.5+\delta}}\lesssim \|v_t\|_{H^{1.5+\delta}}\|b_0\|_{H^{1.5+\delta}}\lesssim \|b_0\|_{H^{1.5+\delta}}\|v\|_{H^{2.5+\delta}}.$$ 
Therefore,
\begin{equation}\label{moresb}
\begin{aligned}
&~~~~\|a^{\nu\alpha}\p_{\nu}b_0^{\mu}\p_t\p_{\mu}b_{\alpha}\|_{H^{0.5+\delta}}+\|\p_t(\p_{\beta}b_{\gamma}\p_{\nu}b_{\alpha})a^{\nu\gamma}a^{\beta\alpha}\|_{H^{0.5+\delta}}+\|\p_{\beta}b_0^{\mu}a^{\beta\alpha}\p_t\p_{\mu}b_{\alpha}\|_{H^{0.5+\delta}}\\
&\lesssim\|a\|_{H^{1.5+\delta}}\|b_0\|_{H^{2.5+\delta}}\|b_t\|_{H^{1.5+\delta}}+\|a\|_{H^{1.5+\delta}}^2\|b\|_{H^{2.5+\delta}}\|b_t\|_{H^{1.5+\delta}}\\
&\lesssim\PP_0+\PP.
\end{aligned}
\end{equation}

Summing these bounds up, and absorbing the $\epsilon$-term to LHS, we obtain: 
\begin{equation}
\|Q_t\|_{H^{2.5+\delta}}\lesssim \PP_0+\PP+P(\|v\|_{H^{2.5+\delta}})\left(\|Q_0\|_{H^{2+\delta}}+\int_0^t\|Q_t\|_{H^{2+\delta}}\right),
\end{equation}
which yields \eqref{estimatesofQt}.

\section{Tangential Estimates}
\label{section 4}
In this section, we establish the tangential energy estimate for the incompressible MHD equations. 
\nota We define $\cp=(I-\overline{\lap})^{1/2}$ where $\overline{\lap}=\p_1^2+\p_2^2$ to be the tangential differential operator. 
\begin{thm} Let $S=\cp^{2.5+\delta}$. Let $E(t)=\|Sv\|_{L^2}^2+\|Sb\|_{L^2}^2+\frac{\epsilon_0}{2}\|a^3_{\alpha}S\eta^{\alpha}\|_{L^2(\Gamma_1)}^2$. Then there exists a $T>0$ such that for each $t\in [0,T]$, the estimate
\begin{equation}\label{T0}
\begin{aligned}
E(t)
\lesssim \PP_0+\int_0^t\PP+\int_0^t P(\|Q\|_{H^{3+\delta}},\|Q_t\|_{H^{2.5+\delta}},\|\eta\|_{H^{3+\delta}})ds
\end{aligned}
\end{equation}
holds. 
\end{thm}

We prove this theorem by estimating $v$ and $b$ separately. 


\subsection{Tangential estimates of $v$}
First, we derive the tangential estimates of $v$.

\begin{equation}\label{V}
\begin{aligned}
\frac{1}{2}\frac{d}{dt}\int_{\Omega}(Sv^{\alpha})(Sv_{\alpha})dy &=\int_{\Omega}(Sv^{\alpha})(\partial_t Sv_{\alpha})  dy \\
&=-\int_{\Omega}(Sv^{\alpha})(S(a^{\mu}_{\alpha}\partial_{\mu}Q))dy+\int_{\Omega}(Sv^{\alpha})(S(b_{\beta}a^{\mu\beta}\partial_{\mu}b_{\alpha}))dy\\
&=: I+J.
\end{aligned}
\end{equation}

To control $I$, we have: 

\begin{equation}
\begin{aligned}\label{I}
I&=-\int_{\Omega}(Sv^{\alpha})(S(a^{\mu}_{\alpha}\p_{\mu}Q))dy\\
&=\underbrace{-\int_{\Omega}(Sv^{\alpha})(a^{\mu}_{\alpha})(S\p_{\mu}Q)dy}_{I_1}\underbrace{-\int_{\Omega}(Sv^{\alpha})(Sa^{\mu}_{\alpha})(\p_{\mu}Q)dy}_{I_2}\\ 
&~~~~~\underbrace{-\int_{\Omega} (Sv^{\alpha})[S(a^{\mu}_{\alpha}\p_{\mu}Q)-a^{\mu}_{\alpha}(S\p_{\mu}Q)-(Sa^{\mu}_{\alpha})\p_{\mu}Q]dy}_{I_3}.
\end{aligned}
\end{equation}
\paragraph*{Control of $I_3$:  }
This is a direct consequence of the Kato-Ponce inequality \eqref{kato3}, i.e., 
\begin{equation}
\begin{aligned}
I_3&\leq \|Sv\|_{L^2}(\|a^{\mu}_{\alpha}\|_{W^{1,6}}\|\p_{\mu}Q\|_{W^{1.5+\delta,3}}+\|a^{\mu}_{\alpha}\|_{W^{1.5+\delta,3}}\|\p_{\mu}Q\|_{W^{1,6}})  \\
&\lesssim \|v\|_{H^{2.5+\delta}}\|a\|_{H^{2+\delta}}\|Q\|_{H^{3+\delta}}\\
&\lesssim \|v\|_{H^{2.5+\delta}}\|\eta\|_{H^{3+\delta}}^2\|Q\|_{H^{3+\delta}}.
\end{aligned}
\end{equation}
\paragraph*{Control of $I_1$:  }
We integrate $\partial_{\mu}$ by parts to get: 
\begin{equation}\label{I10}
\begin{aligned}
I_1&=-\int_{\Omega}Sv^{\alpha}a^{\mu}_{\alpha}(\p_{\mu}SQ)dy \\
&=\int_{\Omega}a^{\mu}_{\alpha}S\p_{\mu}v^{\alpha}(SQ)dy+\underbrace{\int_{\Gamma_0}(SQ)(a_{\alpha}^{3}Sv^{\alpha}) dS(\Gamma_0)}_{=0}-\int_{\Gamma_1}(\underbrace{SQ}_{=0})(a^{\mu}_{\alpha}Sv^{\alpha}N_{\mu}) dS(\Gamma_1)\\
&=\int_{\Omega}\underbrace{S(a^{\mu}_{\alpha}\p_{\mu}v^{\alpha})}_{=0}(SQ)dy-\int_{\Omega}(Sa^{\mu}_{\alpha})\p_{\mu}v^{\alpha}(SQ)dy-\int_{\Omega}[S(a^{\mu}_{\alpha}\p_{\mu}v^{\alpha})-(Sa^{\mu}_{\alpha})\p_{\mu}v^{\alpha}-a^{\mu}_{\alpha}S\p_{\mu}v^{\alpha}](SQ)dy,
\end{aligned}
\end{equation}
where the boundary integrals in the second line vanish since $a^3_1=a^3_2=0$ and $v_3=0$ on $\Gamma_0$, and $\TP Q=\TP(c^2/2)=0$ on $\Gamma_1$. 
The last term in the third line is controlled using \eqref{kato3}: 
\begin{equation}
\begin{aligned}
&~~~~-\int_{\Omega}[S(a^{\mu}_{\alpha}\p_{\mu}v^{\alpha})-(Sa^{\mu}_{\alpha})\p_{\mu}v^{\alpha}-a^{\mu}_{\alpha}S\p_{\mu}v^{\alpha}](SQ)dy \\
&\lesssim (\|a^{\mu}_{\alpha}\|_{W^{1.5+\delta,3}}\|\p_{\mu}v^{\alpha}\|_{W^{1,3}}+\|a^{\mu}_{\alpha}\|_{W^{1,6}}\|\p_{\mu}v^{\alpha}\|_{H^{1.5+\delta}})\|SQ\|_{L^3} \\
&\lesssim\|Q\|_{H^{3+\delta}}\|a\|_{H^{2+\delta}}\|v\|_{H^{2.5+\delta}}\lesssim\|Q\|_{H^{3+\delta}}\|\eta\|_{H^{3+\delta}}^2\|v\|_{H^{2.5+\delta}}.
\end{aligned}
\end{equation}
For the second term in the last line of \eqref{I10}, we need to integrate $1/2$-tangential derivatives by parts and then apply  \eqref{product}: 
\begin{equation}
\begin{aligned}
-\int_{\Omega}Sa^{\mu}_{\alpha}\p_{\mu}v^{\alpha}SQdy&=\int_{\Omega}\TP^{2+\delta}a^{\mu}_{\alpha} \TP^{0.5}(SQ\p_{\mu}v^{\alpha})\\
&\lesssim\|a\|_{H^{2+\delta}}(\|SQ\|_{H^{0.5}}\|\p_{\mu}v^{\alpha}\|_{L^{\infty}}+\|SQ\|_{L^3}\|\p_{\mu}v^{\alpha}\|_{W^{0.5,6}}) \\
&\lesssim\|\eta\|_{H^{3+\delta}}^2\|Q\|_{H^{3+\delta}}\|v\|_{H^{2.5+\delta}}.
\end{aligned}
\end{equation}
Summing these up, we have: 
\begin{equation}\label{I1}
I_1\lesssim\|\eta\|_{H^{3+\delta}}^2\|Q\|_{H^{3+\delta}}\|v\|_{H^{2.5+\delta}}.
\end{equation}

\paragraph*{Control of $I_2$:  } Let $S_m: =-(I-\TL)^{0.25+0.5\delta}\p_m$. Then one may decompose $S$ as: 
\begin{equation}\label{sm}
\begin{aligned}
S&=((I-\TL)^{1.25+0.5\delta}-(I-\TL)^{0.25+0.5\delta})+\underbrace{(I-\TL)^{0.25+0.5\delta}}_{=: S_0}\\
&=(I-\TL)^{0.25+0.5\delta}(-\TL)+S_0\\
&=: \sum_{m=1}^2S_m\p_m+S_0.
\end{aligned}
\end{equation}
 Plugging this decomposition and the identity (which is obtained by differentiating $a: \partial\eta=I$)
\begin{equation}\label{aeta}
\p_m a_{\alpha}^\mu=-a^{\mu}_{\nu}\p_{\beta}\p_m \eta^\nu a^{\beta}_{\alpha}
\end{equation}
into $I_2$, we have: 
\begin{equation}\label{I20}
\begin{aligned}
I_{2}&=-\sum_{m=1}^2\int_{\Omega}(S v^{\alpha})(S_m\p_ma^{\mu}_{\alpha})(\p_{\mu}Q)dy\underbrace{-\int_{\Omega}(Sv^{\alpha})S_0a^{\mu}_{\alpha}\partial_{\mu}Q dy}_{R_1}\\
&=\sum_{m=1}^2\int_{\Omega}(S v^{\alpha})S_m(a^{\mu}_{\nu}\p_{\beta}\p_m\eta^{\nu}a^{\beta}_{\alpha})\p_{\mu}Q dy +R_1\\
&=\underbrace{\sum_{m=1}^2\int_{\Omega}(S v^{\alpha})(S_m\p_{\beta}\p_m\eta^{\nu})(a^{\mu}_{\nu}a^{\beta}_{\alpha})\p_{\mu}Q dy}_{I_{21}}+\\&\int_{\Omega}(S v^{\alpha})[S_m(a^{\mu}_{\nu}\p_{\beta}\p_m\eta^{\nu}a^{\beta}_{\alpha})-(S_m\p_{\beta}\p_m\eta^{\nu})(a^{\mu}_{\nu}a^{\beta}_{\alpha})]\p_{\mu}Qdy+R_1 
\end{aligned}
\end{equation}
Here, $R_1$ is bounded by $P(\|\eta\|_{H^{2.5+\delta}})\|Q\|_{H^{1.5}}\|v\|_{H^{2.5+\delta}}$ via the multiplicative Sobolev inequality, while the last term in the third line of \eqref{I20} can be controlled by using Kato-Ponce inequality \eqref{kato2} as: 
\begin{equation}
\begin{aligned}
&~~~~\int_{\Omega}(S v^{\alpha})[S_m(a^{\mu}_{\nu}\p_{\beta}\p_m\eta^{\nu}a^{\beta}_{\alpha})-(S_m\p_{\beta}\p_m\eta^{\nu})(a^{\mu}_{\nu}a^{\beta}_{\alpha})]\p_{\mu}Qdy \\
&\lesssim (\|a^{\mu}_{\nu}a^{\beta}_{\alpha}\|_{W^{1,6}}\|\p_{\beta}\p_m\eta^{\nu}\|_{W^{0.5+\delta,3}}+\|\p_{\beta}\p_m\eta^{\nu}\|_{L^6}\|a^{\mu}_{\nu}a^{\beta}_{\alpha}\|_{W^{1.5+\delta,3}})\|\p_{\mu}Q\|_{L^{\infty}}\|Sv^{\alpha}\|_{L^2} \\
&\lesssim \|a\|_{H^{2+\delta}}\|a\|_{H^{1.5+\delta}}\|\eta\|_{H^{3+\delta}}\|Q\|_{H^{2.5+\delta}}\|v\|_{H^{2.5+\delta}} 
\lesssim P(\|\eta\|_{H^{3+\delta}})\|Q\|_{H^{2.5+\delta}}\|v\|_{H^{2.5+\delta}}.
\end{aligned}
\end{equation} 

It remains to control $I_{21}$. Writing ${\sum}_{m=1}^2S_m\p_m = S-S_0$, we have: 
\begin{equation}\label{I210}
I_{21}=\int_{\Omega}(S v^{\alpha})(S\p_{\beta}\eta^{\nu})(a^{\mu}_{\nu}a^{\beta}_{\alpha})(\p_{\mu}Q) dy-\int_{\Omega}(S v^{\alpha})(S_0\p_{\beta}\eta^{\nu})(a^{\mu}_{\nu}a^{\beta}_{\alpha})(\p_{\mu}Q)dy.
\end{equation}
It is easy to see the second term in \eqref{I210} can be bounded by $\|v\|_{H^{2.5+\delta}}\|Q\|_{H^{1.5}} P(\|\eta\|_{H^{2.5+\delta}})$. For the first term, we integrate $\p_\beta $ by parts to obtain: 
\begin{equation}
\begin{aligned}\label{I21}
I_{21}
&=\underbrace{-\int_{\Omega}(\p_{\beta}S v^{\alpha})(S\eta^{\nu})(a^{\mu}_{\nu}a^{\beta}_{\alpha})(\partial_{\mu}Q) dy}_{I_{211}} -\int_{\Omega}(S v^{\alpha})(S\eta^{\nu})(\p_{\beta}a^{\mu}_{\nu})a^{\beta}_{\alpha}(\p_{\mu}Q) dy \\
&~~~~-\int_{\Omega}(S v^{\alpha})(S\eta^{\nu})(a^{\mu}_{\nu}a^{\beta}_{\alpha})(\p_{\beta}\p_{\mu}Q) dy+\underbrace{\int_{\Gamma_0}(Sv^{\alpha})(S\eta^{\nu})a^{\mu}_{\nu}a^{\beta}_{\alpha}(\p_{\mu}Q)N_{\beta}dS(\Gamma_0)}_{=0} \\
&~~~~+\underbrace{\int_{\Gamma_1}(Sv^{\alpha})(S\eta^{\nu})a^{\mu}_{\nu}a^{\beta}_{\alpha}(\partial_{\mu}Q)N_{\beta}dS(\Gamma_1)}_{I_{212}} + R_2\\
&\lesssim I_{211}+\|\p a\|_{L^6}\|S\eta\|_{L^3}\|a\|_{L^{\infty}}\|\p Q\|_{L^{\infty}}\|Sv\|_{L^2}+\|a\|_{L^{\infty}}\|S\eta\|_{L^3}\|a\|_{L^{\infty}}\|\p^2Q\|_{L^6}\|Sv\|_{L^2}\\
&~~~~+I_{212}+\|a\|_{L^6}\|\eta\|_{H^{2.5+\delta}}\|a\|_{L^{\infty}}\|\p Q\|_{L^{\infty}}\|Sv\|_{L^2} \\
&\lesssim I_{211}+I_{212}+\PP+P(\|Q\|_{H^3}),
\end{aligned}
\end{equation}
where the integral on $\Gamma_0$ vanishes because $N=(0,0,-1)$ and $a^3_1=a^3_2=0$ on $\Gamma_0$.

Now, we bound $I_{211}$ by the Kato-Ponce commutator estimate \eqref{kato3}, because we want to move the derivatives on $v$ to $a$ in order to control $v$.

\begin{equation}\label{I2110}
\begin{aligned}
I_{211}&=-\int_{\Omega}(S\p_{\beta} v^{\alpha}a^{\beta}_{\alpha})(a^{\mu}_{\nu}S\eta^{\nu})(\p_{\mu}Q) dy \\
&=\int_{\Omega}(\p_{\beta} v^{\alpha})Sa^{\beta}_{\alpha}(a^{\mu}_{\nu}S\eta^{\nu})(\p_{\mu}Q) dy \\
&~~~~+\int_{\Omega}(a^{\mu}_{\nu}S\eta^{\nu}\partial_{\mu} Q)[S(a^{\beta}_{\alpha}\partial_{\beta}v^{\alpha})-(Sa^{\beta}_{\alpha})\p_{\beta} v^{\alpha}-a^{\beta}_{\alpha}S(\partial_{\beta} v^{\alpha})] dy.
\end{aligned}
\end{equation}
The term on the second line of \eqref{I2110} is controlled by \eqref{product} after integrating $0.5$ derivatives by parts, i.e.,
\begin{equation}
\begin{aligned}
\int_{\Omega}(\p_{\beta} v^{\alpha})Sa^{\beta}_{\alpha}(a^{\mu}_{\nu}S\eta^{\nu})(\p_{\mu}Q) dy&=\int_{\Omega}\TP^{1/2}(S\eta^{\nu}a^{\mu}_{\nu}\p_{\mu}Q\p_{\beta} v^{\alpha}) \TP^{2+\delta}a^{\beta}_{\alpha}dy \\
&\lesssim \|a\|_{H^{2+\delta}}\|\TP^{1/2}(S\eta^{\nu}a^{\mu}_{\nu}\p_{\mu}Q\p_{\beta} v^{\alpha})\|_{L^2} \\
&\lesssim \|a\|_{H^{2+\delta}}(\|a S\eta\|_{L^3}\|\p Q\p v\|_{W^{1/2,6}}+\|a S\eta\|_{H^{1/2}}\|\p Q\p v\|_{L^{\infty}}) \\
&\lesssim P(\|\eta\|_{H^{3+\delta}})\|v\|_{H^{2.5+\delta}}\|Q\|_{H^{2.5+\delta}}
\end{aligned}
\end{equation}
In addition, we apply \eqref{kato3} to the term on the third line of \eqref{I2110} and get: 
\begin{equation}
\begin{aligned}
&~~~~\int_{\Omega}(a^{\mu}_{\nu}S\eta^{\nu}\partial_{\mu} Q)[S(a^{\beta}_{\alpha}\partial_{\beta}v^{\alpha})-(Sa^{\beta}_{\alpha})\p_{\beta} v^{\alpha}-a^{\beta}_{\alpha}S(\partial_{\beta} v^{\alpha})] dy\\
&\lesssim \|a\|_{L^{\infty}}\|S\eta\|_{L^3}\|\p Q\|_{L^{\infty}}(\|a\|_{W^{1,6}}\|\p v\|_{H^{1.5+\delta}}+\|a\|_{W^{1.5+\delta,3}}\|\p v\|_{W^{1,3}})\\
&\lesssim  P(\|\eta\|_{H^{3+\delta}})\|v\|_{H^{2.5+\delta}}\|Q\|_{H^{2.5+\delta}}
\end{aligned}
\end{equation}
Therefore,
\begin{equation}\label{I211}
I_{211}\lesssim P(\|\eta\|_{H^{3+\delta}})\|v\|_{H^{2.5+\delta}}\|Q\|_{H^{2.5+\delta}}.
\end{equation}

Now we come to control $I_{212}$.  We shall compute its time integral, which then allows us to integrate $\partial_t$ by parts to eliminate $0.5$ more derivatives falling on $v$. 
Since $N=(0,0,1)$ and $Q=\frac{1}{2}c^2$ on $\Gamma_1$, we have $a_\alpha^\beta N_\beta=a_\alpha^3$ and $a_\nu^\mu \p_\mu Q= a_\nu^3 \p_3 Q$, and so: 

\begin{equation}
\begin{aligned}
\int_0^t I_{212}ds&=\int_0^t\int_{\Gamma_1}(\partial_t S \eta^{\alpha})(S\eta^{\nu})a^{3}_{\nu}a^{3}_{\alpha}(\partial_{3}Q)dS(\Gamma_1)ds \\
&=\frac{1}{2}\int_{\Gamma_1}\underbrace{(S\eta^{\alpha})(S\eta^{\nu})a^{3}_{\nu}a^{3}_{\alpha}}_{\geq 0}(\partial_{3}Q)dS(\Gamma_1) \Big|^{t}_0 \\\\
&~~~~-\int_0^t\int_{\Gamma_1}\partial_t a^{3}_{\alpha}a^3_{\nu}S\eta^{\alpha}S\eta^{\nu}\partial_3 Q dS(\Gamma_1)ds-\frac{1}{2}\int_0^t\int_{\Gamma_1}a^{3}_{\nu}a^{3}_{\alpha}S\eta^{\alpha}S\eta^{\nu}\partial_3Q_t dS(\Gamma_1)ds\\
\end{aligned}
\end{equation}
Invoking the physical sign condition $\partial_3 Q\leq -\epsilon_0$ and Sobolev trace lemma, we have: 
\begin{equation}
\begin{aligned}\label{I212}
\int_0^t I_{212}ds&\leq  -\frac{\epsilon_0}{2}\int_{\Gamma_1}(S\eta^{\alpha})(S\eta^{\nu})a^{3}_{\nu}a^{3}_{\alpha}dS(\Gamma_1)\Big|^{t}_0 \\
&~~~~+\int_0^t \|a_t\|_{H^{1.5+\delta}}\|a\|_{H^{1.5+\delta}}\|\eta\|_{H^{3+\delta}}^2\|Q\|_{H^{2.5+\delta}}ds\\
&~~~~+\int_0^t \|a\|_{H^{1.5+\delta}}^2\|\eta\|_{H^{3+\delta}}^2\|Q_t\|_{H^{2.5+\delta}}ds \\
&\leq -\frac{\epsilon_0}{2}\|a^3_{\alpha}S\eta^{\alpha}\|_{L^2(\Gamma_1)}^2\\
& +P(\|v_0\|_{H^{2.5+\delta}},\|b_0\|_{H^{2.5+\delta}})+\int_0^t  P(\|\eta\|_{H^{3+\delta}},\|Q\|_{H^{2.5+\delta}},\|Q_t\|_{H^{2.5+\delta}}) ds.
\end{aligned}
\end{equation}

Summing up \eqref{I}, \eqref{I1}, \eqref{I20}, \eqref{I21}, \eqref{I211}, \eqref{I212}, we obtain: 
\begin{equation} \label{II}
\int_0^t I(s) ds+\frac{\epsilon_0}{2}\|a^3_{\alpha}S\eta^{\alpha}\|_{L^2(\Gamma_1)}^2\lesssim\PP_0+\int_0^t\PP+\int_0^t P(\|\eta\|_{H^{3+\delta}},\|Q\|_{H^{3+\delta}},\|Q_t\|_{H^{2.5+\delta}})ds.
\end{equation} 

\paragraph*{Control of $J$:  }Now we start to control $J$. We first plug the identity \eqref{GW1} into $J$, then write $J$ to be the sum of the highest order term and the commutator, which again can be controlled by Kato-Ponce inequality \eqref{KATO}

\begin{equation}\label{J}
\begin{aligned}
J&=\int_{\Omega}(Sv^{\alpha})(S(b_{\beta}a^{\mu\beta}\partial_{\mu}b_{\alpha}))dy= \int_{\Omega}(Sv^{\alpha})(S(b_0^{\mu}\p_{\mu}b_{\alpha}))dy \\
&=\underbrace{\int_{\Omega}(Sv^{\alpha})b_0^{\mu}S\partial_{\mu}b_{\alpha} dy}_{J_1}+\int_{\Omega}Sv^{\alpha}[S(b_0^{\mu}\partial_{\mu}b_{\alpha})-b_0^{\mu}S\partial_{\mu}b_{\alpha}S\partial_{\mu}b_{\alpha}] dy \\
&\lesssim J_1+\|v\|_{H^{2.5+\delta}}(\|\p b_0^{\mu}\|_{L^{\infty}}\|\TP^{1.5+\delta}\p_{\mu}b_{\alpha}\|_{L^2}+\|Sb_0^{\mu}\|
_{L^2}\|\p_{\mu}b_{\alpha}\|_{L^{\infty}}) \\
&\lesssim J_1+\|v\|_{H^{2.5+\delta}}\|b_0\|_{H^{2.5+\delta}}\|b\|_{H^{2.5+\delta}}.
\end{aligned}
\end{equation}

The term $J_1$ cannot be controlled directly, but it actually cancels with the highest order term in the energy of $b$. We will see that in the next step.

\subsection{Tangential estimates of $b$}
We derive the tangential estimates of $b$ in this subsection and then conclude the tangential energy estimates. Taking the time derivative of $\frac{1}{2}\|S b\|_{L^2}^2 $ and invoking the identity \eqref{GW1} and Kato-Ponce inequality \eqref{kato2}, we have: 
\begin{equation}\label{K}
\begin{aligned}
\frac{1}{2}\frac{d}{dt}\|S b\|_{L^2}^2 &=\int_{\Omega}(S b_{\alpha})S(b_{\beta}a^{\mu\beta}\partial_{\mu}v^{\alpha}) dy= \int_{\Omega}(S b_{\alpha})S(b_0^{\mu}\partial_{\mu}v^{\alpha}) dy\\
&=\underbrace{\int_{\Omega}(S b_{\alpha})b_0^{\mu}(S\partial_{\mu}v^{\alpha}) dy}_{K_1}+\int_{\Omega}S b_{\alpha}[S(b_0^{\mu}\partial_{\mu}v^{\alpha})-b_0^{\mu}(S\partial_{\mu}v^{\alpha}) ] dy\\
&\lesssim K_1+\|v\|_{H^{2.5+\delta}}\|b_0\|_{H^{2.5+\delta}}\|b\|_{H^{2.5+\delta}}.
\end{aligned}
\end{equation}

Now we are able to see that $J_1$ cancels $K_1$:  Integrating $\partial_{\mu}$ in $J_1+K_1$ by parts, we have
\begin{equation}\label{J1K1}
\begin{aligned}
J_1+K_1&=\int_{\Omega}(Sv^{\alpha})b_0^{\mu}S\partial_{\mu}b_{\alpha} dy+\int_{\Omega}(S b_{\alpha})b_0^{\mu}S\partial_{\mu}v^{\alpha} dy \\
&=\int_{\Omega}\partial_{\mu}(Sv^{\alpha}S b_{\alpha})b_0^{\mu} dy \\
&=-\int_{\Omega}Sv^{\alpha}Sb_{\alpha}\underbrace{\partial_{\mu}b_0^{\mu}}_{\dive b_0=0} dy+\int_{\partial\Omega}Sv^{\alpha}S b_{\alpha}\underbrace{b_{\beta}a^{\mu\beta} N_{\mu}}_{B\cdot N=0}dS(y)=0.
\end{aligned}
\end{equation}

Combining \eqref{V}, \eqref{II}, \eqref{J}, \eqref{K}, \eqref{J1K1}, we derive the tangential estimate as follows: 
\begin{equation}\label{T}
\begin{aligned}
&~~~~\|Sv\|_{L^2}^2+\|Sb\|_{L^2}^2+\frac{\epsilon_0}{2}\|a^3_{\alpha}S\eta^{\alpha}\|_{L^2(\Gamma_1)}^2 \\
&\lesssim P(\|v_0\|_{H^{2.5+\delta}},\|b_0\|_{H^{2.5+\delta}})+\int_0^t P(\|\eta\|_{H^{3+\delta}},\|v\|_{H^{2.5+\delta}},\|b\|_{H^{2.5+\delta}},\|Q\|_{H^{3+\delta}},\|Q_t\|_{H^{2.5+\delta}})ds \\
&\lesssim \PP_0+\int_0^t\PP+\int_0^t P(\|Q\|_{H^{3+\delta}},\|Q_t\|_{H^{2.5+\delta}},\|\eta\|_{H^{3+\delta}})ds
\end{aligned}
\end{equation} which implies in \eqref{T0}.

\begin{flushright}
$\square$
\end{flushright}
\section{Closing the estimates}
\label{section 5}
In this section we close our a priori estimate and prove the physical sign condition can be propagated to a positive time if holds for the initial data. 

\subsection{The div-curl type estimates}
\paragraph*{$H^{2.5+\delta}$-estimate of $v$ and $b$: }
In this subsection we do the div-curl type estimate of $v$ and $b$ to derive the control of full $H^{2.5+\delta}$ norms. Although for Euler equations one can use the Cauchy invariance to give linear estimates for $\curl v$ and $\dive v$,  there is no such analogue for MHD equations. Instead, inspired by Gu-Wang \cite{gu2016construction}, we can derive the evolution equations of $\curl v$  to control the $\curl v$ and $\curl b$ simultaneously thanks to the identity $b=(b_0\cdot\p)\eta$.  Then we apply the div-curl estimate to derive the control of full $H^{2.5+\delta}$ norms of $v$ and $b$. 

The following notations will be adopted throughout: 
\nota Let $X=(X^1,X^2,X^3)$ be a vector field. We denote the ``curl operator" and the ``div operator" in the Eulerian coordinate by 
$$(B_a X)_{\lambda}=\epsilon_{\lambda\tau\alpha}a^{\mu\tau}\p_{\mu}X^{\alpha},\q\text{and}\,\,A_a X=a^{\mu}_{\alpha}\partial_{\mu} X^{\alpha},$$
respectively, where $\epsilon_{\lambda\tau\alpha}$ is the sign of the permutation $(\lambda\tau\alpha)\in S_3$.

\begin{prop}\label{divcurl} For sufficiently small $T>0$, the following estimates hold: 
\begin{equation}
\begin{aligned}
\|\curl v\|_{H^{1.5+\delta}}+\|\curl b\|_{H^{1.5+\delta}}&\lesssim\epsilon(\| v\|_{H^{2.5+\delta}}+\|b\|_{H^{2.5+\delta}}) +\PP_0+\int_0^t \PP ; \\
\|\dive v\|_{H^{1.5+\delta}}+\|\dive b\|_{H^{1.5+\delta}}&\lesssim\epsilon(\| v\|_{H^{2.5+\delta}}+\|b\|_{H^{2.5+\delta}}),
\end{aligned}
\end{equation}
whenever $t\in [0,T]$. 
\end{prop}

\begin{proof} The divergence estimates are easy because $A_av=0$ and $A_a b=0$, so: 
$$\|\dive v\|_{H^{1.5+\delta}}=\|\underbrace{A_a v}_{=0}+(A_I-A_a)v\|_{H^{1.5+\delta}}\lesssim \epsilon\|v\|_{H^{2.5+\delta}};$$
$$\|\dive b\|_{H^{1.5+\delta}}=\|\underbrace{A_a b}_{=0}+(A_I-A_a)b\|_{H^{1.5+\delta}}\lesssim \epsilon\|b\|_{H^{2.5+\delta}}.$$

The estimates for $\|\curl v\|_{H^{1.5+\delta}}$ and $\|\curl b\|_{H^{1.5+\delta}}$ are more dedicate. Since
\begin{equation}\label{curl0}
\begin{aligned}
&~~~~\|\curl v\|_{H^{1.5+\delta}}+\|\curl b\|_{H^{1.5+\delta}}\\
&\leq \|(B_I-B_a)v\|_{H^{1.5+\delta}}+\|(B_I-B_a)b\|_{H^{1.5+\delta}}+\|B_av\|_{H^{1.5+\delta}}+\|B_a b\|_{H^{1.5+\delta}}\\
& \lesssim\epsilon(\| v\|_{H^{2.5+\delta}}+\|b\|_{H^{2.5+\delta}})+\|B_av\|_{H^{1.5+\delta}}+\|B_a b\|_{H^{1.5+\delta}},
\end{aligned}
\end{equation} 
and so it suffices to control $\|B_av\|_{H^{1.5+\delta}}$ and $\|B_a b\|_{H^{1.5+\delta}} $. As mentioned in the beginning of this subsection, we will derive the evolution equation for $B_a v$ and $B_a b$:  Plugging $b_{\beta}a^{\mu\beta}=b_0^{\mu}$ and $b_{\alpha}=(b_0\cdot\p)\eta$ in the first equation of \eqref{MHDL}, and then applying the $\curl$operator $B_a$ on both sides, we have: 
\begin{equation}\label{curl1}
(B_a\p_tv)_{\lambda}=(B_a((b_0\cdot\p)^2\eta))_{\lambda}.
\end{equation}
Commuting $\p_t$ and $b_0\cdot\p$ with $B_a$ on both sides of \eqref{curl1}, we have: 
\begin{equation}
\p_t(B_a v)_{\lambda}-(b_0\cdot\p)B_a((b_0\cdot\p)\eta)_{\lambda}=\epsilon_{\lambda\tau\alpha}\p_ta^{\mu\tau}\p_{\mu}v^{\alpha}+[B_a,b_0\cdot\p]((b_0\cdot\p)\eta)_{\lambda}.
\end{equation}

Taking $\p^{1.5+\delta}$ derivatives, and then commuting it with $\p_t$ and $b_0\cdot\p$, respectively, we get the evolution equation of $B_a v$: 
\begin{equation}\label{curl2}
\p_t(\p^{1.5+\delta}B_a v)_{\lambda}-(b_0\cdot\p)(\p^{1.5+\delta}B_a(b_0\cdot\p)\eta)_{\lambda}=F_{\lambda},
\end{equation}
where 
\begin{equation}\label{curlrhs}
F_{\lambda}=[\p^{1.5+\delta},b_0\cdot\p](B_a(b_0\cdot\p)\eta)_{\lambda}+\p^{1.5+\delta}(\epsilon_{\lambda\tau\alpha}\p_ta^{\mu\tau}\p_{\mu}v^{\alpha}+[B_a,b_0\cdot\p]((b_0\cdot\p)\eta)_{\lambda}).
\end{equation}
Taking the $L^2$ inner product of $\p^{1.5+\delta} B_av$ and \eqref{curl2}, we have: 
\[
\frac{1}{2}\frac{d}{dt}\int_{\Omega}|\p^{1.5+\delta} B_a v|^2dy-\int_{\Omega}\p^{1.5+\delta} B_a v\cdot (b_0^{\nu}\p_{\nu})(\p^{1.5+\delta}B_a (b_0\cdot\p)\eta)dy=\int_{\Omega}F\cdot \p^{1.5+\delta}B_a vdy.
\] 
Integrating $\p_\nu$ by parts in the second term on LHS, commuting $(b_0\cdot\p)$ with $\p^{1.5+\delta}B_a$ and then invoking $\p_t\eta=v$, we have: 
\begin{equation}
\begin{aligned}
\frac{1}{2}\frac{d}{dt}\int_{\Omega}|\p^{1.5+\delta} B_a v|^2+|\p^{1.5+\delta} B_a (b_0\cdot\p)\eta|^2dy
=\underbrace{\int_{\Omega}F\cdot \p^{1.5+\delta}B_a vdy}_{\mathcal{B}_1}\\
\underbrace{+\int_{\Omega}\p^{1.5+\delta} (B_a(b_0\cdot\p)\eta)\cdot [\p^{1.5+\delta}B_a,b_0\cdot\p]vdy }_{\mathcal{B}_2}\\
~~~~\underbrace{+\int_{\Omega}\p^{1.5+\delta}(B_a(b_0\cdot\p)\eta)^{\lambda}\p^{1.5+\delta}(\epsilon_{\lambda\tau\alpha}\p_ta^{\mu\tau}\p_{\mu}(b_0\cdot\p\eta^{\alpha}))dy}_{\mathcal{B}_3},
\end{aligned}
\end{equation}where the boundary term vanishes since $b_0\cdot N=0$ on the boundary.
The control of $\mathcal{B}_3$ is straightforward by the multiplicative Sobolev inequality, say,
\begin{equation}\label{B3}
\mathcal{B}_3\lesssim \|b\|_{H^{2.5+\delta}}^2\|a\|_{H^{1.5+\delta}}\|a_t\|_{H^{1.5+\delta}}\lesssim \|b\|_{H^{2.5+\delta}}^2\|v\|_{H^{2.5+\delta}}\|\eta\|_{H^{2.5+\delta}}^6.
\end{equation}
To control $\mathcal{B}_2$, it suffices to control $\|[\p^{1.5+\delta}B_a,b_0\cdot\p]v\|_{L^2}$. We simplify the commutator term as follows: 
\begin{equation}\label{B20}
\begin{aligned}
[\p^{1.5+\delta}B_a,b_0\cdot\p]v&=\epsilon_{\lambda\tau\alpha}\left(\p^{1.5+\delta}(a^{\mu\tau}\p_{\mu}(b_0^{\nu}\p_{\nu}v^{\alpha}))-b_0^{\nu}\p_{\nu}\p^{1.5+\delta}(a^{\mu\tau}\p_{\mu}v^{\alpha})\right) \\
&=\epsilon_{\lambda\tau\alpha}\underbrace{\left(\p^{1.5+\delta}(a^{\mu\tau}\p_{\mu}(b_0^{\nu}\p_{\nu}v^{\alpha}))-\p_{\nu}\p^{1.5+\delta}(b_0^{\nu}a^{\mu\tau}\p_{\mu}v^{\alpha})\right)}_{\mathcal{B}_{21}}  \\
&~~~~+\epsilon_{\lambda\tau\alpha}\underbrace{\left(\p_{\nu}\p^{1.5+\delta}(b_0^{\nu}a^{\mu\tau}\p_{\mu}v^{\alpha})-b_0^{\nu}\p_{\nu}\p^{1.5+\delta}(a^{\mu\tau}\p_{\mu}v^{\alpha})\right)}_{\mathcal{B}_{22}}.
\end{aligned}
\end{equation}
Invoking the Kato-Ponce commutator estimate \eqref{KATO}, we can control $\mathcal{B}_{22}$ as 
\begin{equation}\label{B22}
\begin{aligned}
&~~~~\|\p_{\nu}\p^{1.5+\delta}(b_0^{\nu}a^{\mu\tau}\p_{\mu}v_{\alpha})-b_0^{\nu}\p_{\nu}\p^{1.5+\delta}(a^{\mu\tau}\p_{\mu}v_{\alpha})\|_{L^2}\\
&\lesssim\|b_0\|_{H^{2.5+\delta}}\|a^{\mu\tau}\p_{\mu}v_{\alpha}\|_{L^{\infty}}+\|\p b_0\|_{L^{\infty}}\|a^{\mu\tau}\p_{\mu}v_{\alpha}\|_{H^{1.5+\delta}} \\
&\lesssim \|b_0\|_{H^{2.5+\delta}}\|v\|_{H^{2.5+\delta}}\|\eta\|_{H^{2.5+\delta}}^2.
\end{aligned}
\end{equation}

For $\mathcal{B}_{21}$, we have 
\begin{equation}\label{B210}
\begin{aligned}
\mathcal{B}_{21}&=\epsilon_{\lambda\tau\alpha}\p^{1.5+\delta}(a^{\mu\tau}\p_{\mu}(b_0^{\nu}\p_{\nu}v^{\alpha}))-\p_{\nu}(b_0^{\nu}a^{\mu\tau}\p_{\mu}v^{\alpha}))\\
&=\epsilon_{\lambda\tau\alpha}\p^{1.5+\delta}\left(a^{\mu\tau}\p_{\mu}b_0^{\nu}\p_{\nu}v^{\alpha}+a^{\mu\tau}b_0^{\nu}\p_{\mu}\p_{\nu}v^{\alpha}-b_0^{\nu}\p_{\nu}a^{\mu\tau}\p_{\mu}v^{\alpha}-b_0^{\nu}a^{\mu\tau}\p_{\mu}\p_{\nu}v^{\alpha}\right) \\
&=\epsilon_{\lambda\tau\alpha}\p^{1.5+\delta}\left(a^{\mu\tau}\p_{\mu}b_0^{\nu}\p_{\nu}v^{\alpha}+b_0^{\nu}\p_{\beta}\p_{\nu}\eta_{\gamma}a^{\mu\gamma}a^{\beta\tau}\p_{\mu}v^{\alpha} \right)\\
&=\epsilon_{\lambda\tau\alpha}\p^{1.5+\delta}(a^{\mu\tau}\p_{\mu}b_0^{\nu}\p_{\nu}v^{\alpha}+\p_{\beta}((b_0\cdot\p)\eta_{\gamma})a^{\mu\gamma}a^{\beta\tau}\p_{\mu}v^{\alpha}-\underbrace{\p_{\beta}b_0^{\nu}\p_{\nu}\eta_{\gamma}a^{\mu\gamma}a^{\beta\tau}\p_{\mu}v^{\alpha}}_{=\p_\beta b_0^\nu \delta_\nu^\mu a^{\beta\tau}\p_\mu v^\alpha} ),
\end{aligned}
\end{equation}where we used \eqref{aeta} to expand $b_0^\nu \p_\nu a^{\mu\tau}\p_\mu v^\alpha$ in the second line. 
Therefore, invoking $b=(b_0\cdot\p)\eta$ again, the $L^2$ norm of $\mathcal{B}_{21}$ can be controlled by the multiplicative Sobolev inequality: 
\begin{equation}\label{B21}
\begin{aligned}
\|\mathcal{B}_{21}\|_{L^2}&\lesssim\|a^{\mu\tau}\p_{\mu}b_0^{\nu}\p_{\nu}v_{\alpha}\|_{H^{1.5+\delta}}+\|\p_{\beta}b_{\gamma}a^{\mu\gamma}a^{\beta\tau}\p_{\mu}v_{\alpha}\|_{H^{1.5+\delta}}+\|\p_{\beta}b_0^{\mu}a^{\beta\tau}\p_{\mu}v_{\alpha}\|_{H^{1.5+\delta}} \\
&\lesssim P(\|\eta\|_{H^{2.5+\delta}})(\|b_0\|_{H^{2.5+\delta}}+\|b\|_{H^{2.5+\delta}})\|v\|_{H^{2.5+\delta}}.
\end{aligned}
\end{equation}

It remains to control $\mathcal{B}_1$, specifically, we need to bound $\|F\|_{L^2}$ given by \eqref{curlrhs}.
The first term is controlled by using Kato-Ponce commutator estimate \eqref{KATO}. Silimarly as in \eqref{B20}, we have
\begin{equation}\label{F1}
\begin{aligned}
\|[\p^{1.5+\delta},b_0\cdot\p](B_a(b_0\cdot\p)\eta)\|_{L^2}&=\|\p^{1.5+\delta}\p_{\nu}(b_0^{\nu}B_a b)-b_0\p^{1.5+\delta}\p_{\nu}B_ab\|_{L^2} \\
&\lesssim\|\p b_0\|_{L^{\infty}}\|B_a b\|_{H^{1.5+\delta}}+ \|b_0\|_{H^{2.5+\delta}}\|B_a b\|_{L^{\infty}}\\
&\lesssim P(\|\eta\|_{H^{2.5+\delta}})\|b_0\|_{H^{2.5+\delta}}\|v\|_{H^{2.5+\delta}}.
\end{aligned}
\end{equation}

For the commutator term in \eqref{curlrhs}, we can proceed similarly as in \eqref{B210} to get
\begin{equation}\label{F2}
\|[B_a,b_0\cdot\p]((b_0\cdot\p)\eta)\|_{H^{1.5+\delta}}\lesssim P(\|\eta\|_{H^{2.5+\delta}})\|b_0\|_{H^{2.5+\delta}}\|v\|_{H^{2.5+\delta}}.
\end{equation}The remaining term in $F$ can be easily bounded by $P(\|\eta\|_{H^{2.5+\delta}})\|b_0\|_{H^{2.5+\delta}}\|v\|_{H^{2.5+\delta}}$ via the multiplicative Sobolev inequality. 

Combining \eqref{B20}, \eqref{B22}, \eqref{B21}, \eqref{F1} and \eqref{F2}, we have 
\begin{equation}\label{curla}
\|B_a v\|_{H^{1.5+\delta}}+\|B_a b\|_{H^{1.5+\delta}}\lesssim\PP_0+\|b_0\|_{H^{2.5+\delta}}\int_0^t \PP~ .
\end{equation}

Therefore, invoking Lemma \ref{estimatesofa} (7), then absorbing the $\epsilon$-term to LHS, we ends the proof by: 
\begin{equation}\label{curl}
\begin{aligned}
&~~~~\|\curl v\|_{H^{1.5+\delta}}+\|\curl b\|_{H^{1.5+\delta}}\\
&\leq \|(B_I-B_a)v\|_{H^{1.5+\delta}}+\|(B_I-B_a)b\|_{H^{1.5+\delta}}+\|B_av\|_{H^{1.5+\delta}}+\|B_a b\|_{H^{1.5+\delta}} \\
& \lesssim\epsilon(\| v\|_{H^{2.5+\delta}}+\|b\|_{H^{2.5+\delta}})+\PP_0+\int_0^t \PP .
\end{aligned}
\end{equation}
\end{proof}

Now we can derive the estimate of full $H^{2.5+\delta}$ derivative estimate of $v$ and $b$. First applying Hodge's decomposition inequality, we get
\begin{equation}\label{vH2.50}
\|v\|_{H^{2.5+\delta}}\lesssim \|v\|_{L^2}+\|\curl v\|_{H^{1.5+\delta}}+\|\dive v\|_{H^{1.5+\delta}}+\|(\TP v)\cdot N\|_{H^{1+\delta}(\Gamma_1)},\\
\end{equation}
For the tangential term, we apply Sobolev trace lemma to get: 
\begin{equation}\label{tangentv}
\|\TP v\cdot N\|_{H^{1+\delta}(\Gamma_1)}\lesssim \|\TP^{1.5+\delta} v_3\|_{H^{0.5}(\Gamma_1)}\lesssim \|\TP^{1.5+\delta} \partial v_3\|_{L^2},
\end{equation}
where the last term in \eqref{tangentv} can be expressed using the tangential derivative of $v$ by: 
\begin{equation}
\partial_3v_3=\dive v-\partial_1v_1-\partial_2v_2=(A_I-A_a)v-\partial_1 v_1-\partial_2 v_2.
\end{equation}
Hence, 
\begin{equation}\label{tangentv2}
\|\TP^{1.5+\delta}\p v_3\|_{L^2}\leq \|\TP^{2.5+\delta} v\|_{L^2}+\|v\|_{H^{0.5}}+\epsilon\|v\|_{H^{2.5+\delta}},
\end{equation}

Combining  \eqref{divcurl} and \eqref{tangentv2}, and then absorbing $\epsilon\|v\|_{H^{2.5+\delta}}$ to the LHS,  we have : 
\begin{equation}\label{vH2.5}
\|v\|_{H^{2.5+\delta}}\lesssim \PP_0+\int_0^t \PP~ds+\|S v\|_{L^2}.
\end{equation}

The estimate of $\|b\|_{H^{2.5+\delta}}$ can be derived exactly in the same way as $\|v\|_{H^{2.5+\delta}}$, so we omit the details.
\begin{equation}\label{bH2.5}
\|b\|_{H^{2.5+\delta}}\lesssim \PP_0+\int_0^t \PP~ds+\|S b\|_{L^2}.
\end{equation}

In conclusion, we have proved
\begin{thm} The following estimates hold in a sufficiently short time interval $[0,T]$: 
\begin{equation}\label{vbH2.5}
\|v\|_{H^{2.5+\delta}}+\|b\|_{H^{2.5+\delta}}\lesssim \PP_0+\int_0^t \PP~ds+\|S v\|_{L^2}+\|S b\|_{L^2}.
\end{equation}
\end{thm}
\begin{flushright}
$\square$
\end{flushright}

\paragraph*{$H^{3+\delta}$-estimate of $\eta$: }
We derive the $H^{3+\delta}$ estimate for $\eta$ via the standard div-curl estimate: 
\begin{equation}\label{etaH3.50}
\|\eta\|_{H^{3+\delta}}\lesssim \|\eta\|_{L^2}+\|\curl \eta\|_{H^{2+\delta}}+\|\dive \eta\|_{H^{2+\delta}}+\|(\TP\eta)\cdot N\|_{H^{1.5+\delta}(\partial\Omega)}.
\end{equation}

The divergence part is easy to treat owing to the div-free condition $A_a v=0$, i.e., the Eulerian divergence of $v$ is identically zero.
\begin{equation}\label{diveta0}
\begin{aligned}
\|\dive \eta\|_{H^{2+\delta}}&\lesssim \|\dive \p\eta\|_{H^{1+\delta}}+\|\dive \eta\|_{H^{1+\delta}} \\
&\lesssim \|A_a\p\eta\|_{H^{1+\delta}}+\|(A_I-A_a)\p\eta\|_{H^{1+\delta}}+\|\eta\|_{H^{2+\delta}}\\
&\lesssim \|A_a\p\eta\|_{H^{1+\delta}}+\epsilon\|\eta\|_{H^{3+\delta}}+\|\eta(0)\|_{H^{2+\delta}}+\int_0^t\|v\|_{H^{2+\delta}}. \\
\end{aligned}
\end{equation}

Now it remains to control $A_a\p\eta$. We have: 
$$ A_a\p\eta(t)=A_a\p\eta(0)+\int_0^tA_{a_t}\p\eta+A_a\p v=\dive\p\eta(0)+\int_0^tA_{a_t}\p\eta+\underbrace{\p(A_av)}_{A_a v=0}-A_{\p a}v~ds.$$
Therefore, it can be controlled as
\begin{equation}
\begin{aligned}\label{diveta1}
\|A_a\p\eta(t)\|_{H^{1+\delta}}&\leq \|\dive\p\eta(0)\|_{H^{1+\delta}}+\int_0^t\|A_{a_t}\p\eta\|_{H^{1+\delta}}+\|A_{\p a}v\|_{H^{1+\delta}}ds\\
&\lesssim \|\eta(0)\|_{H^{3+\delta}}+\int_0^t\|a_t\|_{H^{1.5+\delta}}\|\eta\|_{H^{3+\delta}}+\|a\|_{H^{2+\delta}}\|v\|_{H^{2.5+\delta}}ds\\
&\lesssim \|\eta(0)\|_{H^{3+\delta}}+\int_0^t\|\eta\|_{H^{3+\delta}}\|v\|_{H^{2.5+\delta}}ds.
\end{aligned}
\end{equation}

Summing up \eqref{diveta0} and \eqref{diveta1}, then absorbing the $\epsilon$-term to LHS, we get the control of $\dive \eta$: 
\begin{equation}\label{diveta}
\|\dive \eta\|_{H^{2+\delta}}\lesssim  \|\eta(0)\|_{H^{3+\delta}}+\int_0^t P(\|\eta\|_{H^{3+\delta}},\|v\|_{H^{2.5+\delta}})ds.
\end{equation}

For the boundary estimate, we have: 
\begin{equation}
\begin{aligned}\label{bdryeta}
\|(\cp \eta)\cdot N\|_{H^{1.5+\delta}(\Gamma_1)}&\lesssim \|S\eta\cdot N\|_{L^2(\Gamma_1)}+\|\eta\cdot N\|_{H^{1.5+\delta}(\Gamma_1)}\\
&\lesssim \|a^{3}_\alpha S\eta^\alpha\|_{L^2(\Gamma_1)}+\|(\delta^3_\alpha-a^3_\alpha)S \eta^\alpha\|_{L^2(\Gamma_1)}+\|\eta\|_{H^{2+\delta}}\\
&\lesssim_{\epsilon_0} \frac{\epsilon_0}{2}\|a^3_{\alpha}S\eta^{\alpha}\|_{L^2(\Gamma_1)} + \epsilon \|\eta\|_{H^3}+\|\eta(0)\|_{H^2}+\int_0^t \|v\|_{H^2}.
\end{aligned}
\end{equation}
Here we remark that the term $\frac{\epsilon_0}{2}\|a^3_{\alpha}S\eta^{\alpha}\|_{L^2(\Gamma_1)}$ is exactly the boundary energy term derived from the physical sign condition in the tangential estimate.

It remains to control $\|\curl \eta\|_{H^{2+\delta}}$, we start with
\begin{equation}\label{curleta0}
\begin{aligned}
\|\curl\eta\|_{H^{2+\delta}}&\lesssim \|\curl\partial\eta\|_{H^{1+\delta}}+\|\curl\eta\|_{H^{1+\delta}}\\
&\leq \|B_a\partial\eta\|_{H^{1+\delta}}+\|(B_I-B_a)\partial\eta\|_{H^{1+\delta}}+\|\curl \eta\|_{H^{1+\delta}}.
\end{aligned}
\end{equation}

Recall that the $i$-th component of $B_a\partial\eta$  (resp. $(B_I-B_a)\p\eta$) is of the form $\epsilon_{ijk}a^{\mu j}\p_{\mu}\p\eta^k$ (resp. $\epsilon_{ijk}(\delta^{\mu j}-a^{\mu j})\p_{\mu}\p\eta^k$). So we apply the multiplicative Sobolev inequality \eqref{multiplicative} to get: 
\begin{equation}\label{curleta1}
\|(B_I-B_a)\p \eta\|_{H^{1+\delta}}\leq \|I-a\|_{H^{1.5+\delta}}\|\eta\|_{H^{3+\delta}}\leq\epsilon\|\eta\|_{H^{3+\delta}}.
\end{equation} 
In addition, using multiplicative Sobolev inequality, Young's inequality and Jensen's inequality, we have: 
\begin{equation}\label{curleta2}
\begin{aligned}
\|B_a\p\eta\|_{H^{1+\delta}}&\lesssim\|a\|_{H^{1.5+\delta}}\|\eta\|_{H^{3+\delta}}\lesssim \epsilon^{-1} \|\eta\|_{H^{2.5+\delta}}^4+\epsilon\|\eta\|_{H^{3+\delta}}^2\\
&\lesssim \epsilon^{-1}\|\eta(0)\|_{H^{2.5+\delta}}^4+\epsilon^{-1}\int_0^t \|v\|_{H^{2.5+\delta}}^4 +\epsilon \|\eta\|_{3+\delta}^2
\end{aligned}
\end{equation}
holds for sufficiently small $t$. 
Also,
\begin{equation}\label{curleta3}
\|\curl \eta(t)\|_{H^{1+\delta}}\lesssim \|\eta(t)\|_{H^{2+\delta}}\leq \|\eta(0)\|_{H^{2+\delta}}+\int_0^t\|v\|_{H^{2+\delta}},
\end{equation}
and hence
\begin{equation}\label{curleta}
\|\curl\eta\|_{H^{2+\delta}}\lesssim \epsilon^{-1}P(\|\eta(0)\|_{H^{2.5+\delta}})+\epsilon P(\|\eta\|_{H^{3+\delta}})+\epsilon^{-1}\int_0^tP(\|v\|_{H^{2.5+\delta}}).
\end{equation}

Now summing up \eqref{diveta}, \eqref{bdryeta} and \eqref{curleta}, we get the $H^{3+\delta}$ estimates of $\eta$.

\begin{thm} The following estimates hold in a sufficiently short time interval $[0,T]$: 
\begin{equation}\label{eta3}
\|\eta\|_{H^{3+\delta}} \lesssim_{\epsilon_0} \frac{\epsilon_0}{2}\|a^3_{\alpha}S\eta^{\alpha}\|_{L^2(\Gamma_1)}+\epsilon^{-1}P(\|\eta(0)\|_{H^{2.5+\delta}})+\epsilon P(\|\eta\|_{H^{3+\delta}})+\epsilon^{-1}\int_0^tP(\|v\|_{H^{2.5+\delta}}).
\end{equation}
\end{thm}
\begin{flushright}
$\square$
\end{flushright}

\subsection{Propagation of the physical sign condition}
For the MHD system, we still need to show that the physical sign condition \eqref{taylor0} can be propagated to a positive time if it holds for the initial data, that is, $-\p_3 Q|_{\Gamma_1}\geq \epsilon_0>0$ holds in a short time interval $[0,T]$ for some $\epsilon_0$, provided $-\p_3 Q|_{\Gamma_1}\geq \epsilon_0'>0$ holds at $t=0$ for some $\epsilon_0'$. We start with the following lemma: 
\begin{lem} \label{Igor}
Let $T>0$ be fixed. Assume $f: [0,T]\times \Gamma_1 \to \R$ satisfies $f\in L^{\infty}([0,T]; H^{1.5}(\Gamma_1))$ and $\p_t f\in L^{\infty}([0,T]; H^{0.5}(\Gamma_1))$, then $f\in C^{0,\frac{1}{4}}([0,T]\times \Gamma_1)$.
\end{lem}\label{tl}
\begin{proof}
Since $f\in L^{\infty}([0,T]; H^{1.5}(\Gamma_1))$, we have $\p_1 f, \p_2 f\in L^{\infty}([0,T]; H^{0.5}(\Gamma_1))$. By Sobolev embedding and H\"older's inequality, we have $$L^{\infty}([0,T]; H^{0.5}(\Gamma_1))\hookrightarrow L^{\infty}([0,T]; L^4(\Gamma_1))\hookrightarrow L^4([0,T]; L^4(\Gamma_1))=L^4([0,T]\times\Gamma_1),$$ which implies $f\in W^{1,4}([0,T]\times\Gamma_1).$ Finally, we use Morrey's embedding $W^{1,4}([0,T]\times\Gamma_1)\hookrightarrow C^{0,\frac{1}{4}}([0,T]\times\Gamma_1)$ to conclude that $f\in C^{0,\frac{1}{4}}([0,T]\times \Gamma_1)$.
\end{proof}

Recall we have shown that $Q\in L^{\infty}([0,T]; H^{3+\delta}(\Omega))$ and $Q_t\in L^{\infty}([0,T]; H^{2.5+\delta}(\Omega))$. This, together with the trace lemma, gives $\p_3Q|_{\Gamma_1}\in L^{\infty}([0,T]; H^{1.5}(\Gamma_1))$ and  $\p_3Q_t|_{\Gamma_1}\in L^{\infty}([0,T]; H^{0.5}(\Gamma_1))$. Therefore, we are able to set $f=\p_3Q$ in Lemma \ref{tl} to see that $\p_3 Q$ is 1/4-H\"older continuous in $[0,T]\times \Gamma_1$. Now, suppose $-\p_3 Q|_{\Gamma_1}\geq \epsilon_0'$ holds at $t=0$ for some $\epsilon_0'>0$, then there exists a $\epsilon_0>0$ such that $-\p_3 Q|_{\Gamma_1}>\epsilon_0$  for all $t\in [0,T]$ if the time $T$ is chosen sufficiently small. This verifies that the physical sign condition \eqref{taylor0} can be propagated to a positive time, provided it holds at $t=0$.

\subsection{Gronwall type argument}

Now we recall that 
\begin{equation}
N(t): =\|\eta(t)\|_{H^{3+\delta}}^2+\|v(t)\|_{H^{2.5+\delta}}^2+\|b(t)\|_{H^{2.5+\delta}}^2.
\end{equation}
From \eqref{T0}, \eqref{vbH2.5} and \eqref{eta3}, we have : 
\begin{equation}\label{Gronwall}
\begin{aligned}
N(t)&\lesssim \epsilon P(\|\eta(t)\|_{H^{3+\delta}})++ P(N(0))+P(N(t))\int_0^t P(N(s))ds\\
&~~~~+\epsilon^{-1}P(\|\eta(0)\|_{H^{2.5+\delta}})+\epsilon^{-1}\int_0^t \|v(s)\|_{H^{2.5+\delta}}\,ds.
\end{aligned}
\end{equation}
 For fixed $\epsilon\ll 1$, recall that $\Omega=\mathbb{T}^2\times (0,\beq)$ and $\eta(0)=Id$, one may choose $\beq$ sufficiently small so that $\epsilon^{-1}P(\|\eta(0)\|_{H^{2.5+\delta}}) \leq 1$. Then by a Gronwall-type argument in \cite{tao2006nonlinear} we conclude that: 
\begin{equation}
N(t) \lesssim 1+P(N(0)), \quad \text{when}\q t\in [0,T],
\end{equation}
for some $T=T(N(0), \beq)$. 
\begin{flushright}
$\square$
\end{flushright}

\section{The case of a general domain}
In this section we show how to adapt the ideas used in the proof on Theorem \ref{MHDthm} to the case of a general bounded domain with small volume. The physical situation we have in mind is that of a conducting liquid droplet with sufficiently small volume. We shall adapt the idea used in Section 12 of \cite{DKT} to our case. The goal of this section is to prove:
\thm 
Let $\Omega\subset \mathbb{R}^3$ be a bounded domain with smooth boundary $\Gamma$, and denote by $n$ the unit outward normal to $\Gamma$.  Let $(\eta, v, b)$ be the solution of 
\begin{equation}
\begin{cases}
\partial_tv_{\alpha}-b_{\beta}a^{\mu\beta}\partial_{\mu}b_{\alpha}+a^{\mu}_{\alpha}\partial_{\mu}Q=0~~~& \text{in}~[0,T]\times\Omega;\\
\partial_t b_{\alpha}-b_{\beta}a^{\mu\beta}\partial_{\mu}v_{\alpha}=0~~~&\text{in}~[0,T]\times \Omega ;\\
a^{\mu}_{\alpha}\partial_{\mu}v^{\alpha}=0,~~a^{\mu}_{\alpha}\partial_{\mu}b^{\alpha}=0~~~&\text{in}~[0,T]\times\Omega;\\
a^{\mu\nu}b_{\mu}b_{\nu}=c^2,~Q=\frac{1}{2}c^2,~ a^{\mu}_{\nu}b^{\nu}n_{\mu}=0 ~~~&\text{on}~\Gamma_;
\end{cases}
\end{equation}
 and $\delta\in(0,0.5)$. Assume that $v(0,\cdot)=v_0\in H^{2.5+\delta}(\Omega)$ and $b(0,\cdot)=b_0\in H^{2.5+\delta}(\Omega)$ be divergence free vector fields and $a^{\mu}_{\nu}(0)b_0^\nu n_\mu=0$ on $\Gamma$. Let
\begin{equation}
N(t): =\|\eta(t)\|_{H^{3+\delta}}^2+\|v(t)\|_{H^{2.5+\delta}}^2+\|b(t)\|_{H^{2.5+\delta}}^2.
\end{equation}
Then if $\text{diam}(\Omega):=\bar{\epsilon}\ll 1/8$, there exists a $T>0$, depending only on $N(0)$ and $\beq$ such that $N(t)\leq P(N(0))$ for all $t\in [0,T]$,  provided the physical sign condition
\begin{equation}
-\frac{\p Q}{\p n}\big|_{t=0} \geq \epsilon_0>0,\q\text{on}\,\,\Gamma
\end{equation}
holds.

\paragraph*{Flatten the boundary:} Let $\Omega\subset \mathbb{R}^3$ be a bounded domain with smooth boundary $\Gamma$ with diameter $8\bar{\epsilon}\ll 1$.  Given $y_0\in \Gamma$, there exists $r>0, r \leq  4\bar{\epsilon}$ and a smooth function $\phi$ such that (after a rigid motion and relabeling the coordinates if necessary) we have
$$
\Omega\cap B_r(y_0) = \{y\in B_r(y_0): y_3\geq \phi (y_1,y_2)+1\}.
$$
Now, we take coordinates that flatten the boundary near $y_0$. To be more specific, there exists $R>0$ and a diffeomorphism 
$$
\Phi:\Omega\cap B_r(y_0)\rightarrow  B_R(0,0,1)\cap \{z_3\geq 1\}
$$
such that $\Phi(y_1,y_2,y_3)=(y_1,y_2, y_3-\phi(y_1,y_2))$. Note that $\det (D\Phi)=1$, and so $\det (D\Phi^{-1})=1$. Denoting $\Psi=\Phi^{-1}$ and $\psi=\phi^{-1}$, we have
$$
\Psi(z_1, z_2, z_3)=(z_1,z_2, z_3+\psi(z_1,z_2)). 
$$
Moreover, we must have $R\leq 4\bar{\epsilon}$ since both $\Phi$ and $\Psi$ are volume-preserving diffeomorphisms. 

\paragraph*{The local Lagrangian map and the cut-off functions:} Consider the Lagrangian map $\eta:\Omega\to \Omega(t)$, and set $\tl\eta=\eta\circ \Psi$. Then $\p_t \tl\eta= \p_t\eta\circ \Psi=u\circ \tl\eta$, where $u$ is the velocity of the moving domain $\Omega(t)$. In view of this, if we introduce $$\tl v=u\circ \tl\eta, \q \tl b=B\circ \tl\eta,\q \tl a=[\p \tl\eta]^{-1},\q \tl Q=Q\circ \tl \eta,$$ then these new variables verify the incompressible MHD equations in the domain $B_R(0,0,1)\cap \{z_3\geq 1\}$. We thus use suitably chosen cut-off functions to produce local estimate, passing to the global estimate by the standard gluing procedure. Let $\theta$ be a smooth cut-off function such that $0\leq \theta \leq 1$ with $\theta=1$ in $\bar{B}_{R/5}(0,0,1)$ and $\text{supp}\,\theta\subset B_{R/4}(0,0,1)$. Therefore, extending all quantities to be identically $0$ outside  $B_{R/4}(0,0,1)$ and since $R\leq 4\bar{\epsilon}$, we may consider the equations and variables defined on the reference domain $\tl\Omega=\mathbb{T}^2\times (0,\bar{\epsilon})$. This allows us to adapt the tangential energy estimates in Section \ref{section 4}, but all integrands should carry the cut-off function $\theta$. Also, unlike Section \ref{section 4}, no integral over the lower boundary $\Gamma_0$ of $\tl\Omega$ is present since all variables vanish there in view of the way they have been extended. 

\paragraph*{The energy estimate:}
First, since $\tl\eta(0,z)=(z_1,z_2, z_3+\psi(z_1,z_2))$, a direct computation yields that at $t=0$ we have
\begin{equation*}
\tl a(0)=
\begin{pmatrix}
1&0&-\p_1\psi\\
0&1&-\p_2\psi\\
0&0&1
\end{pmatrix}.
\end{equation*}
In the proof of Theorem \ref{MHDthm}, for which $\psi=0$, we used $a(0)-I=O$, where $O$ stands for the zero matrix, to produce some small parameters (i.e., Lemma \ref{estimatesofa} (7)) in the energy estimates. We need $\p \psi$ to be small in order to apply the same argument here.  This can be achieved since we may assume, without loss of generality, that $\p\psi(0,0)=0$, and so the smallness of $\|\p\psi\|_{L^\infty(\Gamma)}$ can be achieved by the mean value theorem possibly after reducing $\bar{\epsilon}$, provided that $\psi\in H^{2.5+\delta}(\Gamma)$. 

We now apply the energy estimates of Section \ref{section 4} with 
\begin{equation}\label{s}
S\cdot = \cp^{2.5+\delta}(\theta\cdot).
\end{equation}
In order to simplify the exposition, we will omit tildes from all quantities and continue to label $\eta, v, b, a$ and $q$, which are only locally defined Lagrangian flow map, velocity, magnetic field, cofactor matrix, and pressure, respectively.  We start by differentiating $\|S v\|_{L^2(\tl\Omega)}$, i.e., 

\begin{equation}\label{V'}
\begin{aligned}
\frac{1}{2}\frac{d}{dt}\int_{\tl\Omega}(Sv^{\alpha})(Sv_{\alpha}) &=\int_{\tl\Omega}(Sv^{\alpha})(\partial_t Sv_{\alpha})  \\
&=-\int_{\tl\Omega}(Sv^{\alpha})(S(a^{\mu}_{\alpha}\partial_{\mu}Q))+\int_{\tl\Omega}(Sv^{\alpha})(S(b_0^\mu \partial_{\mu}b_{\alpha}))\\
&=: I+J.
\end{aligned}
\end{equation}
To control $I$, we have: 
\begin{equation}
\begin{aligned}\label{I'}
I&=-\int_{\tl\Omega}\cp^{2.5+\delta}(\theta v^{\alpha})\cp^{2.5+\delta}(\theta a^{\mu}_{\alpha}\p_{\mu}Q)\\
&=\underbrace{-\int_{\tl\Omega}\cp^{2.5+\delta}(\theta v^{\alpha}) \theta a^{\mu}_{\alpha}(\cp^{2.5+\delta}[\p_{\mu}Q])}_{I_1}\underbrace{-\int_{\tl \Omega}\cp^{2.5+\delta}(\theta v^{\alpha})\cp^{2.5+\delta}(\theta a^{\mu}_{\alpha})(\p_{\mu}Q)}_{I_2}\\ 
&~~~~~\underbrace{-\int_{\tl \Omega} \cp^{2.5+\delta}(\theta v^{\alpha})[\cp^{2.5+\delta}(\theta a^{\mu}_{\alpha}\p_{\mu}Q)-\theta a^{\mu}_{\alpha}(\cp^{2.5+\delta}[\p_{\mu}Q])-\cp^{2.5+\delta}(\theta a^{\mu}_{\alpha})(\p_{\mu}Q)]}_{I_3}.
\end{aligned}
\end{equation}
Control of $I_1$: We integrate $\p_\mu$ by parts to get
\begin{align}
I_1
=\underbrace{\int_{\tl \Omega}\theta a^{\mu}_{\alpha}(\cp^{2.5+\delta}[\theta\p_{\mu}v^{\alpha}])(\cp^{2.5+\delta}Q)}_{I_{11}}-\int_{\Gamma_1}(\underbrace{\cp^{2.5+\delta}Q}_{=0})(\theta a^{\mu}_{\alpha}Sv^{\alpha}N_{\mu}) dS(\Gamma_1)+\mathcal{R}.
\end{align}
Here and throughout, $\mathcal{R}$ contains error terms when the derivatives fall on $\theta$, which can be controlled by the RHS of \eqref{I1general}. Now,
\begin{align}
I_{11}
=\int_{\tl\Omega}\theta\underbrace{S(a^{\mu}_{\alpha}\p_{\mu}v^{\alpha})}_{=0}(\cp^{2.5+\delta}Q)\underbrace{-\int_{\tl\Omega}\theta(Sa^{\mu}_{\alpha})\p_{\mu}v^{\alpha}(\cp^{2.5+\delta}Q)}_{I_{112}}\nonumber\\
-\underbrace{\int_{\tl\Omega}\theta[S(a^{\mu}_{\alpha}\p_{\mu}v^{\alpha})-(Sa^{\mu}_{\alpha})\p_{\mu}v^{\alpha}-a^{\mu}_{\alpha}S\p_{\mu}v^{\alpha}](\cp^{2.5+\delta} Q)}_{I_{113}}.
\end{align}
$I_{113}$ can be controlled using the Kato-Ponce inequality. To do this, however, each separated term needs to be properly cut-off since the fractional derivatives destroy the compact support. Let $\bar\theta$ be a smooth cut-off function such that $0\leq \bar\theta \leq 1$ with $\text{supp}\,\bar{\theta}\subset B_{R/3}(0,0,1)$ and $\bar{\theta}=1$ on $\text{supp}\,\theta$. The construction of $\bar{\theta}$ allows us to introduce $\bar{\theta}$ without changing given expressions. 
\nota We shall use $C_{\theta}$ to denote constants depend on $||\theta||_{H^{3+\delta}}$ and $||\bt||_{H^{3+\delta}}$ throughout the rest of this section.

Now, commutating $\theta$ through $\cp^{2.5+\delta}$ we get
\begin{equation}
\begin{aligned}\label{similar}
I_{113} &\lesssim \|\cp^{2.5+\delta}[\theta(a^{\mu}_{\alpha}\theta\p_{\mu}v^{\alpha})]-\theta\cp^{2.5+\delta}(a^{\mu}_{\alpha}\theta\p_{\mu}v^{\alpha})\|_{L^{3/2}}\|\bt\cp^{2.5+\delta} Q\|_{L^3}\\
&~~~~+\|\cp^{2.5+\delta}[\theta(a^{\mu}_{\alpha}\theta\p_{\mu}v^{\alpha})]-\cp^{2.5+\delta}(\theta a^{\mu}_{\alpha})\theta\p_{\mu}v^{\alpha}-\theta a^{\mu}_{\alpha}\cp^{2.5+\delta}[\theta\p_{\mu}v^{\alpha}]\|_{L^{3/2}}\|\bt\cp^{2.5+\delta} Q\|_{L^3}.
\end{aligned}
\end{equation}
The first line is bounded by
\begin{equation}
\|\p\theta\|_{L^\infty}\|a^{\mu}_{\alpha}\theta\p_{\mu}v^{\alpha}\|_{W^{1.5+\delta,3/2}}+\|\theta\|_{W^{2.5+\delta,3/2}}\|a^{\mu}_{\alpha}\theta\p_{\mu}v^{\alpha}\|_{L^\infty}\leq C_\theta ||\theta a||_{H^{1.5+\delta}}||\bt v||_{2.5+\delta},
\end{equation}
and the second line is bounded by
\begin{align} \label{similar''}
&~~~~(\|\theta a^{\mu}_{\alpha}\|_{W^{1.5+\delta,3}}\|\theta\p_{\mu}v^{\alpha}\|_{W^{1,3}}+\|\theta a^{\mu}_{\alpha}\|_{W^{1,6}}\|\theta \p_{\mu}v^{\alpha}\|_{H^{1.5+\delta}})\|\bt\cp^{2.5+\delta}Q\|_{L^3}\nonumber\\
 &\leq C_\theta\|\bt Q\|_{H^{3+\delta}}\|\theta a\|_{H^{2+\delta}}^2\|\theta v\|_{H^{2.5+\delta}}.
\end{align}
Moreover, we integrate $1/2$-tangential derivatives by parts and then $I_{112}$ becomes 
\begin{equation} \label{similar'}
\int_{\tl \Omega}\TP^{2+\delta}[\theta a^{\mu}_{\alpha}] \TP^{0.5}(\theta\cp^{2.5+\delta}Q\p_{\mu}v^{\alpha})+\mathcal{R}
\end{equation}
 where
\begin{align}
&~~~~\int_{\tl \Omega}\TP^{2+\delta}[\theta a^{\mu}_{\alpha}] \TP^{0.5}(\theta\cp^{2.5+\delta}Q\p_{\mu}v^{\alpha})\\
&\lesssim \|\theta a\|_{H^{2+\delta}}(\|\theta \cp^{2.5+\delta}Q\|_{H^{0.5}}\|\bt\p_{\mu}v^{\alpha}\|_{L^{\infty}}+\|\theta \cp^{2.5+\delta}Q\|_{L^3}\|\bt \p_{\mu}v^{\alpha}\|_{W^{0.5,6}}) \\
&\leq C_\theta\|\theta a\|_{H^{2+\delta}}\|\theta Q\|_{H^{3+\delta}}\|\bt v\|_{H^{2.5+\delta}}.
\end{align}
Summing these up, we have
\begin{equation}
I_1 \leq C_\theta \Big(\|\theta a\|_{H^{2+\delta}}\|\theta Q\|_{H^{3+\delta}}\|\bt v\|_{H^{2.5+\delta}}+\|\bt Q\|_{H^{3+\delta}}\|\theta a\|_{H^{2+\delta}}^2\|\theta v\|_{H^{2.5+\delta}}\Big). 
\label{I1general}
\end{equation}
\textbf{Control of $I_3$: } We have
\begin{align}
I_3&\lesssim \|\cp^{2.5+\delta}(\theta v)\|_{L^2}\|\cp^{2.5+\delta}(\theta a^{\mu}_{\alpha}\bt\p_{\mu}Q)-\theta a^{\mu}_{\alpha}(\cp^{2.5+\delta}[\bt\p_{\mu}Q])-\cp^{2.5+\delta}(\theta a^{\mu}_{\alpha})(\bt\p_{\mu}Q)\|_{L^2}\nonumber\\
&\leq C_\theta \|\theta v\|_{H^{2.5+\delta}}\|\theta a\|_{H^{2+\delta}}\|\bt Q\|_{H^{3+\delta}}.
\end{align}
\textbf{Control of $I_2$:} First it is easy to check that the decomposition \eqref{sm} remains valid, i.e., for any smooth function $u$, we have
\begin{align}
Su = \sum_{m=1}^2 S_m\p_m(\theta u)+S_0(\theta u),
\end{align}
where $S\cdot$ is defined as \eqref{s}, and $S_m, S_0$ are defined in \eqref{sm}. Then the analysis of \eqref{I20} suggests that it suffices to consider the term associated to $I_{21}$, i.e.,  
$$
I_{21}'=\sum_{m=1}^2\int_{\tl \Omega}\cp^{2.5+\delta}(\theta v^\alpha)[S_m (\theta\p_\beta \p_m \eta^\nu)](a^\mu_\nu a^\beta_\alpha)\p_\mu Q.
$$
Writing $\sum_{m=1}^2S_m \p_m=\cp^{2.5+\delta}-S_0$, we have
\begin{align}
I_{21}' = \int_{\tl \Omega}\cp^{2.5+\delta}(\theta v^\alpha)\cp^{2.5+\delta}(\theta\p_\beta \eta^\nu)(a^\mu_\nu a^\beta_\alpha)\p_\mu Q-\int_{\tl \Omega}\cp^{2.5+\delta}(\theta v^\alpha)S_0(\theta\p_\beta \eta^\nu)(a^\mu_\nu a^\beta_\alpha)\p_\mu Q+\mathcal{R},
\end{align}
where the second term is controlled directly by $C_\theta\|\theta v\|_{H^{2.5+\delta}}\|\bt a\|_{H^{1.5+\delta}}^2\|\bt Q\|_{H^{1.5}}\|\theta \eta\|_{H^{2.5+\delta}}$.
For the first term, we integrate $\p_\beta$ by parts to obtain
\begin{equation}
\begin{aligned}\label{I21'}
I_{21}'
&=\underbrace{-\int_{\tl \Omega}\cp^{2.5+\delta}\p_\beta(\theta v^\alpha)(\cp^{2.5+\delta}\theta\eta^\nu)(a^\mu_\nu a^\beta_\alpha)\p_\mu Q}_{I_{211}'} -\int_{\tl \Omega}\cp^{2.5+\delta}(\theta v^\alpha)(\cp^{2.5+\delta}\theta\eta^\nu)\p_\beta[(a^\mu_\nu a^\beta_\alpha)\p_\mu Q] \\
&~~~~+\underbrace{\int_{\Gamma_1}\cp^{2.5+\delta}(\theta v^{\alpha})(\cp^{2.5+\delta}\theta\eta^{\nu}) a^{\mu}_{\nu}a^{\beta}_{\alpha}(\partial_{\mu}Q)N_{\beta}dS(\Gamma_1)}_{I_{212}'}.
\end{aligned}
\end{equation}
There is no problem to control the second term in the first line of \eqref{I21'} and $I'_{212}$ is controlled analogous to $I_{212}$ in Section \ref{section 4}. For $I_{211}'$, we write
\begin{align}
I_{211}'&=-\int_{\tl \Omega}a^\beta_\alpha S\p_\beta v^\alpha( a^\mu_\nu S\eta^\nu)\p_\mu Q+\mathcal{R}\nonumber\\
&=\int_{\tl\Omega} (S a_{\alpha}^\beta)(\p_\beta v^\alpha)( a^\mu_\nu S\eta^\nu)\p_\mu Q +  \int_{\tl\Omega} [S(a_\alpha^\beta \p_\beta v^\alpha)-(S a_{\alpha}^\beta)(\p_\beta v^\alpha)-a^\beta_\alpha S\p_\beta v^\alpha]( a^\mu_\nu S\eta^\nu)\p_\mu Q.
\end{align}
The first term can be treated similar to \eqref{similar'} by integrating $0.5$-derivatives by parts. The second term is equal to
\begin{align}
&~~~~\int_{\tl\Omega} [\cp^{2.5+\delta}(\theta a_\alpha^\beta \p_\beta v^\alpha)-\cp^{2.5+\delta}(\theta a_{\alpha}^\beta)(\p_\beta v^\alpha)-\theta a^\beta_\alpha \cp^{2.5+\delta}\p_\beta v^\alpha]( a^\mu_\nu S\eta^\nu)\p_\mu Q\nonumber\\
&-\int_{\tl\Omega}a_\alpha^\beta [\cp^{2.5+\delta} (\theta \p_\beta v^\alpha)-\theta \cp^{2.5+\delta}\p_\beta v^\alpha]( a^\mu_\nu S\eta^\nu)\p_\mu Q.
\end{align}
The first line can be controlled similar to \eqref{similar''}, and since
\begin{equation}
\begin{aligned}
&~~~~\|\cp^{2.5+\delta} (\theta \p_\beta v^\alpha)-\theta \cp^{2.5+\delta}\p_\beta v^\alpha\|_{L^2}\\
& \lesssim \|\p\theta\|_{L^\infty} \|v\|_{H^{2.5+\delta}}+ \|\cp^{2.5+\delta}\theta\|_{L^2} \|\p v\|_{L^\infty}
\leq C_\theta \|v\|_{H^{2.5+\delta}}.
\end{aligned}
\end{equation} 
so the second line can be bounded by 
$
C_\theta \|\bt a\|_{H^{1.5+\delta}}^2\|\bt Q\|_{H^{2.5+\delta}}\|\theta\eta\|_{H^{2.5+\delta}}\|\bt v\|_{H^{2.5+\delta}}.
$
\bigskip

\noindent \textbf{Control of $J+\frac{d}{dt}\|Sb\|_{L^2}^2$}: This follows from the what has been done in Section \ref{section 4} except that the cancellation \eqref{J1K1}  holds up to a term of type $\mathcal{R}$, which can still be controlled appropriately. 

\bigskip

After covering $\Gamma$ with finitely many balls, the procedure described above yields the tangential energy estimates near the $\
\Gamma$. We still need to cover the region of $\Omega$ not covered by these balls. However, we have no problem to cover this region using finitely many balls with radius $r\leq 4\bar{\epsilon}$ and again reducing the tangential estimates to $\tl\Omega$. In addition, there are no boundary integrals on either $\Gamma_1$ and $\Gamma_0$. 

Finally, we need to show that the estimates in Section \ref{section 3} and Section \ref{section 5} are still valid in each local coordinate patch. This follows from adapting the estimates in Section \ref{section 3} and Section \ref{section 5} to the MHD equations after commuting $\theta$, i.e., 
\begin{equation}
\begin{cases}
\partial_t (\theta v_{\alpha})-b_{\beta}a^{\mu\beta}\partial_{\mu}(\theta b_{\alpha})+a^{\mu}_{\alpha}\partial_{\mu}(\theta Q)=-b_{\beta}a^{\mu\beta}(\partial_{\mu}\theta) b_{\alpha}+a^{\mu}_{\alpha}(\partial_{\mu}\theta) Q~~~& \text{in}~[0,T]\times\tl\Omega;\\
\partial_t (\theta b_{\alpha})-b_{\beta}a^{\mu\beta}\partial_{\mu}(\theta v_{\alpha})=-b_{\beta}a^{\mu\beta}(\partial_{\mu}\theta) v_{\alpha}~~~&\text{in}~[0,T]\times \tl\Omega ;\\
a^{\mu}_{\alpha}\partial_{\mu} v^{\alpha}=0,~~a^{\mu}_{\alpha}\partial_{\mu}b^{\alpha}=0~~~&\text{in}~[0,T]\times\tl \Omega;\\
a^{\mu\nu}b_{\mu}b_{\nu}=c^2,~Q=\frac{1}{2}c^2,~ a^{\mu}_{\nu}b^{\nu}N_{\mu}=0 ~~~&\text{on}~\Gamma_;\\
\end{cases}\label{MHDtheta}
\end{equation}
We can recover the equations for $Q$, $Q_t$, $B_a v$  and $B_a b$ modulo error terms involving derivatives land on $\theta$, but these contribute only to lower order terms.


\begin{thebibliography}{}

\bibitem[\protect\astroncite{Alazard et~al.}{2014}]{alazard2014cauchy}
Alazard, T., Burq, N., and Zuily, C. (2014).
\newblock On the cauchy problem for gravity water waves.
\newblock {\em Inventiones mathematicae}, 198(1): 71--163.

\bibitem[\protect\astroncite{Bourgain and Li}{2015}]{bourgain2015strong}
Bourgain, J. and Li, D. (2015).
\newblock Strong ill-posedness of the incompressible euler equation in
  borderline sobolev spaces.
\newblock {\em Inventiones mathematicae}, 201(1): 97--157.

\bibitem[\protect\astroncite{Cao and Wu}{2010}]{cao2010two}
Cao, C. and Wu, J. (2010).
\newblock Two regularity criteria for the 3D MHD equations.
\newblock {\em Journal of Differential Equations}, 248(9): 2263--2274.

\bibitem[\protect\astroncite{Chen and Wang}{2008}]{chen2008existence}
Chen, G.-Q. and Wang, Y.-G. (2008).
\newblock Existence and stability of compressible current-vortex sheets in
  three-dimensional magnetohydrodynamics.
\newblock {\em Archive for Rational Mechanics and Analysis}, 187(3): 369--408.

\bibitem[\protect\astroncite{Christodoulou and
  Lindblad}{2000}]{christodoulou2000motion}
Christodoulou, D. and Lindblad, H. (2000).
\newblock On the motion of the free surface of a liquid.
\newblock {\em Communications on Pure and Applied Mathematics},
  53(12): 1536--1602.

\bibitem[\protect\astroncite{Coutand and Shkoller}{2007}]{coutand2007LWP}
Coutand, D. and Shkoller, S. (2007).
\newblock Well-posedness of the free-surface incompressible euler equations
  with or without surface tension.
\newblock {\em Journal of the American Mathematical Society}, 20(3): 829--930.

\bibitem[\protect\astroncite{Disconzi and Ebin}{2014}]{disconzi2014limit}
Disconzi, M.~M. and Ebin, D.~G. (2014).
\newblock On the limit of large surface tension for a fluid motion with free
  boundary.
\newblock {\em Communications in Partial Differential Equations},
  39(4): 740--779.

\bibitem[\protect\astroncite{Disconzi and Ebin}{2016}]{disconzi2016free}
Disconzi, M.~M. and Ebin, D.~G. (2016).
\newblock The free boundary euler equations with large surface tension.
\newblock {\em Journal of Differential Equations}, 261(2): 821--889.

\bibitem[\protect\astroncite{Disconzi and Kukavica}{2019}]{disconzi2017prioriI}
Disconzi, M.~M. and Kukavica, I. (2019).
\newblock A priori estimates for the free-boundary euler equations with surface
  tension in three dimensions.
\newblock {\em Nonlinearity}, 32(9): 3369--3405.

\bibitem[\protect\astroncite{Disconzi et~al.}{2019}]{DKT}
Disconzi, M.~M., Kukavica, I., and Tuffaha, A. (2019).
\newblock A lagrangian interior regularity result for the incompressible free
  boundary euler equation with surface tension.
\newblock {\em SIAM Journal on Mathematical Analysis, 51(5), 3982–4022.}

\bibitem[\protect\astroncite{Friedlander and
  Vicol}{2011}]{friedlander2011global}
Friedlander, S. and Vicol, V. (2011).
\newblock Global well-posedness for an advection--diffusion equation arising in
  magneto-geostrophic dynamics.
\newblock In {\em Annales de l'Institut Henri Poincare (C) Non Linear
  Analysis}, volume~28, pages 283--301. Elsevier Masson.

\bibitem[\protect\astroncite{Friedlander and
  Vishik}{1990}]{friedlander1990nonlinear}
Friedlander, S. and Vishik, M. (1990).
\newblock Nonlinear stability for stratified magnetohydrodynamics.
\newblock {\em Geophysical \& Astrophysical Fluid Dynamics}, 55(1): 19--45.

\bibitem[\protect\astroncite{Friedlander and
  Vishik}{1995}]{friedlander1995stability}
Friedlander, S. and Vishik, M. (1995).
\newblock On stability and instability criteria for magnetohydrodynamics.
\newblock {\em Chaos:  An Interdisciplinary Journal of Nonlinear Science},
  5(2): 416--423.

\bibitem[\protect\astroncite{Gu and Wang}{2016}]{gu2016construction}
Gu, X. and Wang, Y. (2016).
\newblock On the construction of solutions to the free-surface incompressible
  ideal magnetohydrodynamic equations.
\newblock {\em Journal de Math\'ematiques Pures et Appliqu\'ees, Vol. 128: 1-41.}


\bibitem[\protect\astroncite{Hao}{2017}]{hao2016motion}
Hao, C. (2017).
\newblock On the Motion of Free Interface in Ideal Incompressible MHD.
\newblock {\em Archive for Rational Mechanics and Analysis}, Vol. 224: 515-553.

\bibitem[\protect\astroncite{Hao and Luo}{2014}]{hao2014priori}
Hao, C. and Luo, T. (2014).
\newblock A priori estimates for free boundary problem of incompressible
  inviscid magnetohydrodynamic flows.
\newblock {\em Archive for Rational Mechanics and Analysis}, 212(3): 805--847.

\bibitem[\protect\astroncite{Hao and Luo}{2018}]{hao2018ill}
Hao, C. and Luo, T. (2018).
\newblock Ill-posedness of free boundary problem of the incompressible ideal
  MHD.
\newblock {\em  Communications in Mathematical Physics, to appear.}

\bibitem[\protect\astroncite{Ignatova and Kukavica}{2016}]{ignatova2016local}
Ignatova, M. and Kukavica, I. (2016).
\newblock On the local existence of the free-surface euler equation with
  surface tension.
\newblock {\em Asymptotic Analysis}, 100(1-2): 63--86.

\bibitem[\protect\astroncite{Jiang et~al.}{2010}]{jiang2010incompressible}
Jiang, S., Ju, Q., and Li, F. (2010).
\newblock Incompressible limit of the compressible magnetohydrodynamic
  equations with periodic boundary conditions.
\newblock {\em Communications in Mathematical Physics}, 297(2): 371--400.

\bibitem[\protect\astroncite{Kato and Ponce}{1988}]{kato1988commutator}
Kato, T. and Ponce, G. (1988).
\newblock Commutator estimates and the euler and navier-stokes equations.
\newblock {\em Communications on Pure and Applied Mathematics}, 41(7): 891--907.

\bibitem[\protect\astroncite{Kukavica and
  Tuffaha}{2014}]{kukavica2014regularity}
Kukavica, I. and Tuffaha, A. (2014).
\newblock A regularity result for the incompressible euler equation with a free
  interface.
\newblock {\em Applied Mathematics \& Optimization}, 69(3): 337--358.

\bibitem[\protect\astroncite{Kukavica et~al.}{2017}]{kukavica2017local}
Kukavica, I., Tuffaha, A., and Vicol, V. (2017).
\newblock On the local existence and uniqueness for the 3d euler equation with
  a free interface.
\newblock {\em Applied Mathematics \& Optimization}, 76(3): 535--563.

\bibitem[\protect\astroncite{Lindblad}{2002}]{lindblad2002}
Lindblad, H. (2002).
\newblock Well-posedness for the linearized motion of an incompressible liquid
  with free surface boundary.
\newblock {\em Communications on Pure and Applied Mathematics},
  56(02): 153--197.

\bibitem[\protect\astroncite{Lindblad}{2005}]{lindblad2005well}
Lindblad, H. (2005).
\newblock Well-posedness for the motion of an incompressible liquid with free
  surface boundary.
\newblock {\em Annals of mathematics}, pages 109--194.

\bibitem[\protect\astroncite{Lindblad and Luo}{2018}]{lindblad2018priori}
Lindblad, H. and Luo, C. (2018).
\newblock A priori estimates for the compressible euler equations for a liquid
  with free surface boundary and the incompressible limit.
\newblock {\em Communications on Pure and Applied Mathematics}, Vol 71(7): 1273-1333.

\bibitem[\protect\astroncite{Luo}{2018}]{luo2017motion}
Luo, C. (2018).
\newblock On the motion of a compressible gravity water wave with vorticity.
\newblock {\em Annals of PDE}, 4(2): 1--71.

\bibitem[\protect\astroncite{Morando et~al.}{2014}]{morando2014well}
Morando, A., Trakhinin, Y., and Trebeschi, P. (2014).
\newblock Well-posedness of the linearized plasma-vacuum interface problem in
  ideal incompressible MHD.
\newblock {\em Quart. Appl. Math}, 72(3): 549--587.

\bibitem[\protect\astroncite{Nalimov}{1974}]{nalimov1974cauchy}
Nalimov, V. (1974).
\newblock The cauchy-poisson problem.
\newblock {\em Dinamika Splo{\v{s}}n. Sredy,(Vyp. 18 Dinamika Zidkost. so
  Svobod. Granicami)}, 254: 104--210.

\bibitem[\protect\astroncite{Schweizer}{2005}]{SchweizerFreeEuler}
Schweizer, B. (2005).
\newblock On the three-dimensional {E}uler equations with a free boundary
  subject to surface tension.
\newblock {\em Ann. Inst. H. Poincar\'e Anal. Non Lin\'eaire}, 22(6): 753--781.

\bibitem[\protect\astroncite{Secchi and Trakhinin}{2011}]{secchi2011well}
Secchi, P. and Trakhinin, Y. (2011).
\newblock Well-posedness of the linearized plasma-vacuum interface problem.
\newblock {\em arXiv: 1112.3101}.

\bibitem[\protect\astroncite{Secchi and Trakhinin}{2013}]{secchi2013well}
Secchi, P. and Trakhinin, Y. (2013).
\newblock Well-posedness of the plasma--vacuum interface problem.
\newblock {\em Nonlinearity}, 27(1): 105.

\bibitem[\protect\astroncite{Shatah and Zeng}{2008a}]{shatah2008geometry}
Shatah, J. and Zeng, C. (2008a).
\newblock Geometry and a priori estimates for free boundary problems of the
  euler's equation.
\newblock {\em Communications on Pure and Applied Mathematics}, 61(5): 698--744.

\bibitem[\protect\astroncite{Shatah and Zeng}{2008b}]{shatah2008priori}
Shatah, J. and Zeng, C. (2008b).
\newblock A priori estimates for fluid interface problems.
\newblock {\em Communications on Pure and Applied Mathematics}, 61(6): 848--876.

\bibitem[\protect\astroncite{Shatah and Zeng}{2011}]{shatah2011local}
Shatah, J. and Zeng, C. (2011).
\newblock Local well-posedness for fluid interface problems.
\newblock {\em Archive for Rational Mechanics and Analysis}, 199(2): 653--705.

\bibitem[\protect\astroncite{Sun et~al.}{2017}]{sun2017well}
Sun, Y., Wang, W., and Zhang, Z. (2017).
\newblock Well-posedness of the plasma-vacuum interface problem for ideal
  incompressible MHD.
\newblock {\em Archive for Rational Mechanics and Analysis, 234(1): 81-113.}

\bibitem[\protect\astroncite{Tao}{2006}]{tao2006nonlinear}
Tao, T. (2006).
\newblock {\em Nonlinear dispersive equations:  Local and global analysis}.
\newblock Number 106. American Mathematical Soc.

\bibitem[\protect\astroncite{Trakhinin}{2010}]{trakhinin2010well}
Trakhinin, Y. (2010).
\newblock On the well-posedness of a linearized plasma-vacuum interface problem
  in ideal compressible MHD.
\newblock {\em arXiv: 1005.4109}.

\bibitem[\protect\astroncite{Wang}{2012}]{wangyanjin2012}
Wang, Y. (2012).
\newblock Critical magnetic number in the magnetohydrodynamic rayleigh-taylor
  instability.
\newblock {\em Journal of Mathematical Physics}, 53(7).

\bibitem[\protect\astroncite{Wu}{2011}]{wu2011global}
Wu, J. (2011).
\newblock Global regularity for a class of generalized magnetohydrodynamic
  equations.
\newblock {\em Journal of Mathematical Fluid Mechanics}, 13(2): 295--305.

\bibitem[\protect\astroncite{Wu}{1997}]{wu1997LWPww}
Wu, S. (1997).
\newblock Well-posedness in sobolev spaces of the full water wave problem in
  2-D.
\newblock {\em Inventiones mathematicae}, 130(1): 39--72.

\bibitem[\protect\astroncite{Wu}{1999}]{wu1999LWPww}
Wu, S. (1999).
\newblock Well-posedness in sobolev spaces of the full water wave problem in
  3-D.
\newblock {\em Journal of the American Mathematical Society}, 12(2): 445--495.

\bibitem[\protect\astroncite{Zhang and Zhang}{2008}]{zhang2008free}
Zhang, P. and Zhang, Z. (2008).
\newblock On the free boundary problem of three-dimensional incompressible
  euler equations.
\newblock {\em Communications on Pure and Applied Mathematics}, 61(7): 877--940.

\end{thebibliography}
\end{document}